\documentclass[reqno, 11pt]{amsart}

\usepackage[super]{nth}

\usepackage{graphicx}
\usepackage{calligra}
\usepackage{amsfonts, amsthm, amsmath, amssymb}
\usepackage{hyperref}
\hypersetup{colorlinks=false}
\usepackage{bbm}  
\usepackage[margin=1in]{geometry}

\usepackage{helvet}
\usepackage{xcolor} 

 \RequirePackage{mathrsfs} \let\mathcal\mathscr
  
\usepackage{hyperref} 
\usepackage{enumerate}

\usepackage{ dsfont }

\numberwithin{equation}{section}

\newtheorem{theorem}{Theorem}[section] 
\newtheorem{lemma}[theorem]{Lemma}
\newtheorem{proposition}[theorem]{Proposition}
\newtheorem{corollary}[theorem]{Corollary}

\theoremstyle{definition}

 \newtheorem*{acknowledgements}{Acknowledgements}
\newtheorem{remark}[theorem]{Remark}
\newtheorem{definition}[theorem]{Definition}

\renewcommand{\phi}{\varphi}

\renewcommand{\leq}{\leqslant}
\renewcommand{\le}{\leqslant}
\renewcommand{\geq}{\geqslant}
\renewcommand{\ge}{\geqslant}

\newcommand{\x}{\mathbf{x}}

\renewcommand{\c}{\mathbf{c}}

\renewcommand{\b}{\mathbf{b}}

\renewcommand{\r}{\mathbf{r}}

\renewcommand{\l}{\left}

\newcommand{\md}[1]{  \left(\textnormal{mod}\ #1\right)}

\newcommand{\Q}{\mathbb{Q}}
\newcommand{\F}{\mathbb{F}}
\newcommand{\N}{\mathbb{N}}
\newcommand{\R}{\mathbb{R}}

\newcommand{\Z}{\mathbb{Z}}
 
\renewcommand{\r}{\right}
\renewcommand{\b}{\mathbf}
\renewcommand{\c}{\mathcal}
\renewcommand{\epsilon}{\varepsilon}

\renewcommand{\leq}{\leqslant}
\renewcommand{\geq}{\geqslant}
\renewcommand{\#}{\sharp}

\title 
[Serre's problem on statistics of Brauer symbols]
{Serre's problem on statistics of Brauer symbols}

\author{Efthymios Sofos} 
\address{Dipartimento di Matematica\\
Universit{\`a} di Roma
Tor Vergata\\  Italy}
\email{efthymios.sofos@uniroma2.eu}

\date{15 September 2026}


 \dedicatory{\textbf{Au mathématicien aux mille tours}}
\begin{document}
\begin{abstract} We solve  
Serre's problem on the specialisation 
of Brauer group elements in all cases where 
  the number of variables  
is sufficiently large. \end{abstract}  

\vspace{-4cm}
\maketitle
 
\vspace{-0,4cm}

\setcounter{tocdepth}{1}

\tableofcontents

\section{Introduction} \label{s:intro} 
\noindent{\calligra L}\hspace{0.1cm}ate
in the twentieth century, Jean-Pierre Serre 
\cite{serresieve} 
introduced  an important counting function. {\calligra E}\hspace{0.08cm}valuating
the frequency with 
which a Brauer   symbol  
$\alpha=\sum_{i=1}^r (f_i,g_i)$ in $
\mathrm{Br}_2(\mathbb Q(T_1,\ldots,T_n))$ 
vanishes, he defined for $X>0$
the counting function $N_\alpha(X)$ by 
$$   \#\left\{t=(t_1,\ldots,t_n) \in \mathbb Z^n: 
|t_i| \leq X \ \forall i, \ \  \sum_{i=1}^r (f_i(t),g_i(t))=0, \ \ \prod_{i=1}^r f_i(t) g_i(t) \neq 0\right\},$$ where
$f_i,g_i$ are non-zero polynomials in 
$ \mathbb Z[T_1,\ldots,T_n]$.

\noindent
\raisebox{-0.2ex}{\rotatebox{15}{\calligra O}}\hspace{0.05cm}btaining
an upper bound via the large sieve, he   proved  
$N_\alpha(X)=O(X^n (\log X)^{-\Delta(\alpha)})$
for a non-negative constant $\Delta(\alpha)$ 
(see Definition \ref{def_polardiv})
and asked for  asymptotics for $N_\alpha(X)$.

\noindent 
{\calligra N}\hspace{0.01cm}ow, following the 
foundational special-case lower bounds and asymptotics proved by 
Hooley \cite{hoo1,hoo2},
Guo \cite{MR1309229},
Friedlander--Iwaniec \cite{MR2675875},
and Gamburd--Ghosh--Sarnak--Whang \cite{sarn}, 
as well as the influential 
geometric framework 
initiated by Loughran \cite{MR3852186}
and Loughran–Smeets \cite{135714028}, 
we answer Serre's original question in all 
cases with a sufficiently large number of variables:
assume that $n,d,r$ are strictly
positive integers such that $ n>d2^{d+1} r+2r$.
For each $i=1,\ldots, r$ 
let  $f_i,g_i$ be non-zero polynomials in 
$\mathbb Z[T_1,\ldots,T_n]$ whose 
top-degree homogeneous parts all have degree $d$ and 
form a  non-singular system of forms. 
\begin{theorem}\label{thm_main}
 There exists a strictly 
 positive constant $c=c(\b f,\b g )$ such that 
for all $X\geq 2 $ we have  
$$N_\alpha(X) =\sigma(\alpha)
\frac{X^n}{(\log X)^{\Delta(\alpha)}} +O\l(
\frac{X^n}{(\log X)^{\Delta(\alpha)+c}}  \r)
,$$where   $\sigma(\alpha)$ is 
an explicit constant that is described in 
Remark \ref{rem_constant2}.
\end{theorem} A system of   forms 
$F_1,\ldots, F_k \in \mathbb C[T_1,\ldots,T_n]$
is called  {non-singular} if the Jacobian matrix
$(  \partial F_i(\mathbf{x})/\partial T_j  )_{i,j}$ has full
rank  $k$ at every non-zero   point $\mathbf{T} \in \mathbb C^{n} 
\setminus \{\mathbf{0}\}$ 
satisfying $\mathbf{F}( \mathbf{T}) = \mathbf{0}$. 
 \begin{definition}\label{def_polardiv}
 The constant $\Delta(\alpha)$ is defined as 
$\frac{1}{2}d(\alpha)$, where $d(\alpha)$ is the number of irreducible components of $\alpha$.
\end{definition} \begin{remark}
[$\Delta(\alpha)$  in the setting of Theorem \ref{thm_main}]
The ramification locus of $\alpha$ 
consists of all $\b t $ for which 
$\prod_{i=1}^r f_i(\b t ) g_i(\b t )=0$, hence, the smoothness assumption of 
Theorem \ref{thm_main} ensures that the   
irreducible components are 
$$\Big(\bigcup_{i=1}^r\{\b t : f_i(\b t)=0\}\Big)
\bigcup \Big(\bigcup_{i=1}^r\{\b t : g_i(\b t)=0\} \Big).$$ 
In particular, 
 $\Delta(\alpha)=\frac{1}{2}d(\alpha)=r$
 in the setting of Theorem \ref{thm_main}.
\end{remark}

\begin{remark}[Literature comparison] 
The condition $\alpha(\mathbf{t}) = 0$ holds if and only if 
\begin{equation}\label{Or ch'è dal sol difesa}
\prod_{i=1}^r (f_i(\mathbf{t}), g_i(\mathbf{t}))_{\mathbb{R}} = 1 \quad \text{and} \quad \prod_{i=1}^r (f_i(\mathbf{t}), g_i(\mathbf{t}))_{\mathbb{Q}_p} = 1 \ \ 
\text{for every prime} \ \  p, 
\end{equation} where 
$(\cdot, \cdot)_k$ is the local Hilbert symbol 
on $k =\R$ and $k=\Q_p$. When $r=1$
the condition in \eqref{Or ch'è dal sol difesa}
is about counting the number of certain 
Diophantine equations that are everywhere locally soluble.
Problems of this kind   have attracted 
  a lot of  attention recently
(see the introduction of \cite{LRS} for 
a  list of   developments). 
In contrast, the original 
problem of Serre,  regards statistics of Brauer symbols 
and 
does not quite fit into the framework 
of random Diophantine equations. In particular, 
asymptotics were   previously known only 
in two cases: for $r=1$ by work of 
Da Silva \cite[Theorem 4.1.2]{silvaphd}
and for $r> 1$ by work of 
Destagnol--Lyczak--Sofos \cite{11388}
when each $g_i$ is the same constant polynomial.
\end{remark} 

 \begin{remark}[Leading constant]\label{rem_constant2}
In \S \ref{s-prf-thrm23} we will show that 
the leading constant $\sigma(\alpha)$ equals 
the sum  of $2^r$ Euler 
products
$$\sigma(\alpha)=
 \frac{1}{(d\pi)^{\Delta(\alpha)}}
  \sum_{\c A \subset \{1,\ldots, r\}} 
 \sigma_\infty(\c A)  
  \prod_{\substack{ p  \textrm{ prime} \\ p=2}}^\infty 
  \frac{\sigma_p(\c A) }{( 1-\frac{1}{p} )^{\Delta(\alpha)}}
,$$ with 
$$\sigma_v(\c A) =\int_{\substack{\mathbf{t} \in \mathcal{O}_v^n \\ \prod_{i=1}^r (f_i(\mathbf{t}), g_i(\mathbf{t}))_v = 1}} \left( \prod_{i \in \mathcal{A}} (f_i(\mathbf{t}), g_i(\mathbf{t}))_{\Q_v} \right) \mathrm{d}\mu_v(\mathbf{t})
$$ for $v\in \{\infty\}\cup\{p\textrm{ prime}\}$,
where $\c O_\infty=[-1,1]$, $\c O_p=\Z_p$,
 $\mathbb{Q}_\infty=\R$ 
 and $\mu_v$ is the normalised 
 Haar measure on $\mathbb{Q}_v^n$ 
 satisfying $\mu_v(\mathcal{O}_v^n) = 1$.
We will also show that each of the Euler products 
converges absolutely. Furthermore,  in
  \S \ref{s-prf-thrm23}
  we will show that  
$$ \sigma(\alpha)= 2^r \sum_{\substack{ \b s , \b s' \in \{-1,1\}^r \\ 
\prod_{i=1}^r (s_i,s'_i)_\R=1 }}
\gamma(\b s , \b s ')
\mathrm{vol}(\b t \in [-1,1]^n : 
\mathrm{sign}(f_j(\b t)= s_j, 
\mathrm{sign}(g_j(\b t)= s'_j \forall j \in \{1,\ldots, r\}
)),
$$
where $\gamma(\b s , \b s ') $ is given by 
$$
 \lim_{T \to \infty}  
\l[
\prod_{\substack{p \ \mathrm{ prime} \\ p \leq T} }
\left(1-\frac1p\right)^{-\Delta(\alpha)} 
\hspace{-0.2cm} 
\mu_p
\r]
\hspace{-0.1cm}
\l \{\b t \in   
\prod_{p \leq T} \Z_p^{n+1}
\colon \ 
\hspace{-0.2cm}
\begin{array}{l} 
\prod_{i=1}^r (f_i(\b t ),  g_i(\b t ) )_{\Q_p}=1 \ 
\forall p \leq T,  \\ 
\prod_{p\leq T} (f_i(\b t ),  g_i(\b t ) )_{\Q_p}
=(s_i,s'_i)_\R  \  \forall i\in \{1,\ldots,R\}\end{array}
\hspace{-0.1cm}
\right\}
.$$This expression is analogous to  
the leading constant in \cite[Theorem 1.1]{11388}, which arises 
from an underlying Brauer group in a 
fibration structure. 
Here, however, 
$N_\alpha(X)$  
lacks an obvious  fibration setting,
hence, 
it is not apparent if $\sigma(\alpha)$ 
fits in   the Loughran--Smeets framework \cite{LRS}.
\end{remark}
\subsection{Intermezzo}
For a  positive integer $r$ define 
$$ c_r:=\frac{4^{r-1}}{\pi^r}
\sum_{k=0}^r {r\choose k}^2
( 2^{k} + 2^{r-k}   )
 ((2/9)^{k}+(2/9)^{r-k} ) 
\prod_{\substack{ p \ \mathrm{ prime} \\ p=3 } }^\infty 
\frac{ 
\l( \l(  1+1/p \r)^{-2 k}+
\l(  1+1/p\r)^{-2 (r-k)} \r)}{2(1-\frac{1}{p})^r}.
 $$
 We highlight a    
 corollary of our work behind 
Theorem \ref{thm_main}
that will be proved in \S\ref{s-prf-thrm3}.
\begin{theorem}\label{thm_main3}Let $r$ be a 
strictly positive integer. Then as $X\to \infty$ we have 
$$\#\left\{\b n, \b n' \in \mathbb Z^r: 
0<|n_i|,|n'_i|\leq X \ \forall i, \ \  
1=\prod_{i=1}^r (n_i,n'_i)_{\R}=
\prod_{i=1}^r (n_i,n'_i)_{\Q_p} \ 
\forall  \ \mathrm{prime }\  p
\r\}
\sim c_r \frac{X^{2r}}{(\log X)^r}
.$$
\end{theorem} \begin{remark}We shall prove    
$c_r \geq  (4r)^{-1}
(4 \pi / 3)^r$ and the more natural expression
\begin{equation}\label
 {Leonardo Leo_Miserere concertato a due chori}
 c_r= (2/\pi)^r 
\sum_{\substack{ \b s , \b s' \in \{-1,1\}^r \\ 
\prod_{i=1}^r (s_i,s'_i)_\R=1 }}
c_r(\b s , \b s '),\end{equation}
where $c_r(\b s , \b s ') $ is given by 
$$
 \lim_{T \to \infty}  
\l[
\prod_{\substack{p \ \mathrm{ prime} \\ p \leq T} }
\left(1-\frac1p\right)^{-r} 
\hspace{-0.2cm} 
\mu_p
\r]
\hspace{-0.1cm}
\l \{\b n,\b n' \in   
\prod_{p \leq T} \Z_p^{r}
\colon \ \hspace{-0.2cm}
\begin{array}{l} 
\prod_{i=1}^r (n_i,  n'_i )_{\Q_p}=1 \ 
\forall p \leq T,  \\ 
\prod_{p\leq T} (n_i,  n'_i )_{\Q_p}
=(s_i,s'_i)_\R  \  \forall i\in \{1,\ldots,R\}\end{array}
\hspace{-0.1cm}
\right\}.$$
\end{remark}

\subsection{Recommencement}Let us outline the proof of 
of Theorem \ref{thm_main}.
For property \eqref{Or ch'è dal sol difesa} to hold, 
the individual Hilbert symbols $(f_i(\mathbf{t}), g_i(\mathbf{t}))_{\mathbb{Q}_p}$ do not each need 
take prefixed values. It would be simpler  
to deal with individual fixed
values of each Hilbert symbol via the
character sum method of 
  Heath-Brown \cite{rhbII}
and  Fouvry--Kl\"uners \cite{fouklu}. 
Our strategy can be summarised by saying that 
we partition  in progressions modulo 
some $W\to+\infty$   and then 
using the geometric-large sieve  from \cite{MR4961246} 
 to show  that for large primes $p$,  
property \eqref{Or ch'è dal sol difesa}
is the same as  $(f_i(\mathbf{t}), g_i(\mathbf{t}))_{\Q_p}=1$
for exactly one index $i$. In more detail,
\begin {enumerate}
\item 
we use 
a recent circle method tool of 
Destagnol--Lyczak--Sofos \cite[\S 2]{11388} 
that utilizes  Rydin Myerson's work \cite{MR3815565}  
to convert the problem into one where each $f_i$ and $g_i$ is  a linear polynomial in a single variable.
\item We then 
 apply the geometric-large sieve 
 to show that, up to negligible error terms, 
 the integer $\prod_{i=1}^r f_i(t) g_i(t)$ is 
 nearly square-free for almost all $t$ counted by 
 $N_\alpha(X)$. This helps 
 into writing  $N_\alpha(X)$ as a 
product of simpler counting functions.

\item 
We use the character sum method 
to get asymptotics for these simpler counting functions.
\end {enumerate}

\subsection{Generalisation to 
other polynomials}
We prove a statement  analogous of Theorem 
\ref{thm_main} for more general polynomials 
in Theorem \ref{thm:main2}.
In particular, 
the restriction that the top degree parts 
have the same degree 
can be removed  
by  proving that the 
systems of polynomials in the work of 
Browning--Heath-Brown \cite{MR3605019}
are \textit{strong Hardy--Littlewood}.
Furthermore, note that the traditional expectation is that 
a generic system of polynomials is strong Hardy--Littlewood
as long as the number of variables exceeds a linear function of the 
degrees, hence, Theorem \ref{thm:main2} should apply to much more 
general cases than those covered by  Theorem 
\ref{thm_main}.

For integer polynomials 
$F_1,\ldots,F_R$ in $n$ variables  the usual
singular series and singular integral are respectively 
denoted by $\mathfrak{S}$  and $\mathfrak{J}$
and are defined, provided convergence, as  
 \begin{align*}
\mathfrak{S} 
:&= \sum_{q=1}^{\infty} q^{-n} 
\sum_{\substack{a_1, \dots, a_R = 1 \\ 
\gcd(a_1, \dots, a_R, q) = 1}}^{q} 
\mathrm e\Big( -\frac{1}{q}\sum_{i=1}^R a_i  \nu_i \Big) 
\sum_{\b x \in (\Z/q\Z)^{n}} 
\mathrm e\Big( \frac{1}{q}\sum_{i=1}^R a_i F_i(\b x)\Big)  \\
  &=
  \prod_{\substack{ p \textrm{ prime}\\p=2}}^\infty
  \lim_{m\to \infty} 
\frac{\#\{\b t \in (\Z/p^{m}\Z )^{n}: 
\b F (\b t ) = \boldsymbol \nu \}}{p^{m(n-R) }}  
\end{align*}  and for a box $\c B \subset [-1,1]^n$ as 
$$ \mathfrak{J}:= \int_{\boldsymbol \gamma \in \mathbb R^R}
\mathrm e\Big(  -\sum_{i=1}^R \gamma_i \nu_i P^{-d} \Big)
\int_{\boldsymbol \zeta \in  \mathscr{B} } 
\mathrm e\Big( \sum_{i=1}^R  \gamma_i 
F_i^\natural(\boldsymbol \zeta) \Big)
 \mathrm d \boldsymbol \zeta  
\mathrm d \boldsymbol \gamma ,$$ where 
$\mathrm e(z)=\mathrm e^{2 i \pi z}$ and $F^\flat$ denotes the top degree homogeneous part of a polynomial $F$.
Our analytic tool
will apply to the following systems of polynomials.
The following is \cite[Definition 2.2]{11388}:
\begin{definition}[strong 
Hardy--Littlewood system]\label{stronghardylittlew}
 A set of integer polynomials $\{F_1,\ldots,F_R\}$ 
 in $n$ variables is said to be a strong 
Hardy--Littlewood system when the following three properties hold:
\begin{itemize}
\item there exists 
$\eta=\eta(\b F)>0$   such that for all $\b a $ and 
$q$ appearing in $\mathfrak S$ one has 
\begin{equation}\label{eq:singseries}
\left| \sum_{\b x \in (\Z/q\Z)^{n}} 
\mathrm e\left( \frac{1}{q}\sum_{i=1}^R a_i F_i(\b x) \right)
\right| \ll q^{n-1-R-\eta},
\end{equation}  where  the  implied constant 
depends only on $\b F$;
\item  for all  boxes
$\mathscr{B}\subset [-1,1]^{n}$ whose 
  side   length is at most $1$  
 and  all $\boldsymbol \gamma \in \mathbb R^R$ 
one has
\begin{equation}\label{eq:singint} \left| 
\int_{\boldsymbol \zeta \in  \mathscr{B} } 
\mathrm e\Big( \sum_{i=1}^R  \gamma_i 
F_i^\natural(\boldsymbol \zeta) \Big)
 \mathrm d \boldsymbol \zeta  \right| \ll 
 (1+ \max_i | \gamma_i|)^{-R-\eta},
\end{equation} 
where the  implied constant   depends only on $\b F $ 
and $\c B$, and \item 
there is $\delta=\delta(\b F )>0$  
such that for all  boxes
$\mathscr{B}\subset [-1,1]^{n}$ whose 
  side   length is at most $1$    and all 
 $P \ge 1$, $\boldsymbol \nu \in \mathbb{Z}^R$ we have  
\begin{equation}\label{eq_simon_asymptotic}
 \#\{\mathbf{x} \in \mathbb{Z}^{n}\cap  P \mathscr{B}: \mathbf{F}(\mathbf{x}) = \boldsymbol \nu\}= \mathfrak{J} \mathfrak{S}
 P^{n-dR} + O(P^{n-dR-\delta}),
\end{equation} where  the  implied constant 
depends only on $\b F$ and $\mathscr{B}$.
\end{itemize} \end{definition}

\begin{theorem}\label{thm:main2}Assume that 
 $f_i,g_i\in\mathbb Z[T_1,\ldots,T_n]$ are 
 such that the system 
$\{f_1,g_1,\ldots, f_r,g_r\}$ is strong Hardy--Littlewood. 
 Then  there exists a strictly 
 positive constant $c=c(\b f,\b g )$ such that 
for all $X\geq 2 $ we have  
$$N_\alpha(X) =\sigma(\alpha)
\frac{X^n}{(\log X)^{\Delta(\alpha)}} +O\l(
\frac{X^n}{(\log X)^{\Delta(\alpha)+c}}  \r)
,$$where   $\sigma(\alpha)$ 
is an explicit constant that is described in 
Remark \ref{rem_constant2}.
\end{theorem}
In light of \cite[Corollary 2.6]{11388},
Theorem \ref{thm_main}
is a direct corollary  of 
Theorem \ref{thm:main2}.

\begin{acknowledgements} 
We   dedicate this work to Jean-Pierre Serre 
on the  occasion of his 100\textsuperscript{th}
birthday.
While our contribution is modest, 
it is offered in deep admiration of 
his profound influence on mathematics,
both through the groundbreaking methods 
that reshaped entire fields and 
through his beautifully crafted lecture notes.
\end{acknowledgements}

\section{Geometric-large sieve}

Assume that $\Omega$ is a subset of $\mathbb Z^n$
and for $m$ a strictly positive integer define 
$$ \omega ( p) :=1-\frac{ \#\Omega (\Z/p^m \Z)}{p^{nm}}$$ 
for any prime $p$. For $t\geq 1 $ let 
$$ L(t) := \sum_{q\leq t} \mu(q)^2 
\prod_{p\mid q }\frac{\omega(p)}{1-\omega(p)}.$$ 
Serre \cite[page 401]{serresieve} gave a version of the large sieve that bounds the density of $\Omega$ by essentially $1/L(t)$ for some appropriate $t$. 
Now let $Y$ be any closed subscheme of 
$\mathbb A_{\b Z}^n$
of codimension $k\geq 2 $. Bhargava's geometric sieve 
\cite[Theorem 3.3]{Bhargava} bounds the density of 
integers $\b x $ 
that fall in $Y(\mathbb F_p)$ for some prime $p>z$ by 
essentially $1/z^{k-1}$. In recent work, 
Pagano--Sofos \cite[Theorem 1.10]{MR4961246}
proved a common generalisation of these results that gives a saving compared to applying only one of these sieves.
This was proved for sets that are bounded by the multiple 
of a bounded area. 
Here we provide a   generalisation that holds 
for sets that are not homogeneously expanding, 
allowing for lopsided boxes. \begin{lemma}
[lopsided geometric-large sieve]\label{lem_lopsided}
Assume that  $m,n,\Omega,\omega,L, Y,k$ are as above and
that $\limsup_{p\to\infty} \omega(p)$ is not $1$. 
Then for any $B_1,\ldots, B_n\geq 1 $ we have  
\begin{align*} 
&\#\left\{\b x \in   \Omega \cap \prod_{i=1}^n 
[-B_i,B_i]:\b x \md{p^m}
\in Y(\Z/p^m\Z)\ \textrm{ for some } p> z\right\}\\
& \ll \frac{\prod_{i=1}^n B_i }
{L({\min_i B_i}^{1/(4m)})z^{k-1} \log z} 
+\frac{(\max_i B_i)^n}{(\min_i B_i)^{(k-1)/(4m)} 
\log \min_i B_i}
+ (\max_i B_i)^{n-k+1}
,\end{align*} where   the implied constant 
depends only on $m,n,Y$ and $\limsup_{p\to\infty} \omega(p)$.\end{lemma}
\begin{proof}Let $B=\max B_i$.  Since $m\geq 1 $, for 
 any $z_1>z$ we can upper-bound the cardinality in the lemma by  $$
\#\left\{\b x \in \Z^n \cap \prod_{i=1}^n [-B,B]^n:\b x \md{p}\in Y(\F_p) \ \textrm{ for some } p> z_1\right\}
$$ and then apply the geometric sieve
\cite[Theorem 3.3]{Bhargava}
to get 
$$\ll  \frac{B^n}{z^{k-1} \log z}+ B^{n-k+1}.$$ To deal with the contribution of the primes $p \in (z,z_1]$ 
we use the union bound to obtain \begin{align*} 
&\#\left\{\b x \in \Z^n\cap \Omega \cap \prod_{i=1}^n 
[-B_i,B_i]: \b x \md{p^m}
\in Y(\Z/p^m\Z)\ \textrm{ for some } p\in (z,z_1]\right\} \\
\leq   \sum_{p\in (z,z_1]} 
& \#\left\{\b x \in \Z^n \cap  \Omega \cap 
\prod_{i=1}^n [-B_i,B_i]: 
\b x \md{p^m} \in Y(\Z/p^m\Z)  \right\} =:\Xi
.\end{align*} We now apply the lopsided version of Serre's 
large sieve as given by Wilson \cite[Lemma 5.1]{wilson}. 
For   a fixed   $p$ 
and  an arbitrary prime $\ell$ we define  
$ S_\ell \subset (\Z/\ell^m\Z)^n$ via 
\[S_\ell = \begin{cases}
 (\Z/\ell^m\Z)^n \setminus \Omega(\Z/\ell^m\Z), & \text{if } \ell\neq p  \\
 (\Z/p^m\Z)^n \setminus (\Omega(\Z/p^m\Z)\cap Y(\Z/p^m\Z)),  & \text{if } \ell = p.  \end{cases}\]
Then \cite[Lemma 5.1]{wilson} with $M_i=-B_i$,$N_i=2B_i$,
$s=m$ and  $B_0=\min_i B_i$ gives 
$$ \#\left\{\b x \in \Z^n \cap  \Omega \cap 
\prod_{i=1}^n [-B_i,B_i]: 
\b x \md{p^m} \in Y(\Z/p^m\Z)  \right\} \ll 
\frac{B_1 \cdots B_n}{M_p(B_0^{1/(2m)})},$$
where for $t\geq 1 $ we define $$ M_p(t):= \sum_{q\leq t}
\mu(q)^2 \prod_{\ell\mid q}\frac{\#S_\ell }
{\ell^{nm}-\#S_\ell }.$$ This function occurs also in the 
proof of  \cite[Lemma 2.1]{MR4961246}, 
where it is proved that
$$M_p(t) \geq L(\sqrt t)\frac{\# \Omega(\Z/p^m\Z)}
{\# (S(\Z/p^m\Z) \cap \Omega(\Z/p^m\Z))}.$$
Hence, $$ \Xi \ll \frac{B_1 \cdots B_n}{L( {B_0}^{1/(4m)}) }
\sum_{p\in (z,z_1]} 
\frac{\# (S(\Z/p^m\Z) \cap \Omega(\Z/p^m\Z))}
{\# \Omega(\Z/p^m\Z)} .$$ Finally, it was shown in~\cite[\S 2.1]{MR4961246} that under the assumption $\limsup_{p \to \infty} \omega(p) \neq 1$ (which is subsumed by our present hypotheses), the bound
\[
\sum_{p \in (z, z_1]} \frac{\# \left( S(\mathbb{Z}/p^m\mathbb{Z}) \cap \Omega(\mathbb{Z}/p^m\mathbb{Z}) \right)}{\# \Omega(\mathbb{Z}/p^m\mathbb{Z})} \ll \frac{1}{z^{k-1} \log z}
\]
holds for all $z \geq 2$. This estimate is sufficient to complete the proof.\end{proof}
 \begin{remark}
The condition on $\limsup_{p \to \infty} \omega(p)$ serves as the simplest criterion to guarantee that $Y \neq \Omega$. This non-triviality condition is necessary to achieve a saving beyond what can be obtained from the large sieve or the geometric sieve alone.
\end{remark}\begin{remark}
Lemma~\ref{lem_lopsided} is useful primarily when the parameters $B_i$ are of comparable magnitudes. For instance, if $\min_i B_i > (\max_i B_i)^{1-c}$ for a sufficiently small constant $c = c(k,m) > 0$, then both $(\max_i B_i)^n (\min_i B_i)^{-(k-1)/(4m)}$ and $(\max_i B_i)^{n-k+1}$ are negligible compared to $\prod_{i=1}^n B_i$. For example, when 
$m=2$ then if $\max_i B_i \leq (\min B_i)^
{1+\frac{1}{8(n-1)}}$ then 
\begin{equation}\label{eq_papoutsia}
\frac{(\max_i B_i)^n}{(\min_i B_i)^{1/8} }
+(\max_i B_i)^{n-1}
\leq \frac{\prod_{i=1}^n B_i }
{(\min_iB_i )^{1/16}} 
.\end{equation}
Lemma~\ref{lem_lopsided} is   sufficient for most applications; furthermore,  a more general version can be   obtained by   establishing a lopsided analog of the geometric sieve, which could then be substituted directly into the proof of Lemma~\ref{lem_lopsided}  without further modification.
\end{remark}
\begin{lemma}\label{lem_hilbert_symbol_bound}
Let $r \ge 1, m\geq 3,
z\ge 8r^2$ 
be  integers, let $\b s, \b s' \in \{-1,1\}^r$ 
and let  $W_z$ be as
in \eqref{Porpora Torbido intorno al core}.
Let   $x_1, \dots, x_r\geq W_z^2$ and
$ x'_1, \dots, x'_r \ge W_z^2$ be such that  
$W_z \max_i \{x_i,x'_i\} \leq  
(\min \{x_i,x'_i\})^{1+\frac{1}{16r}}$ and let 
$\b a,\b a'\in (\Z/ {W_z}\Z)^r$.
Denote by  $\Gamma_4$  the number of  
$(\mathbf{m}, \mathbf{m}') $ in $\N^r 
\times \N^r $ satisfying 
\begin{enumerate}
    \item $m_i \le x_i$ and $m_i \equiv a_i \pmod {W_z}$ 
    for all $1 \le i \le r$,
    \item $m'_i \le x'_i$ and $m'_i \equiv a'_i \pmod {W_z}$ 
    for all $1 \le i \le r$,
    \item $\prod_{i=1}^r (s_i m_i, s'_i m'_i)_{\mathbb{Q}_p} = 1$ for all primes $p>z$,  
    \item there exists a prime $p > z$ such that $p^2 \mid \prod_{i=1}^r m_i m'_i$  
\end{enumerate} and denote by 
 $\Gamma_3$  the number of  $(\mathbf{m}, \mathbf{m}') $ 
 satisfying the first three properties. Then 
$$\Gamma_3 \ll 
\frac{\prod_{i=1}^{r} x_i x'_i}
{W_z^{2r} (\log  x_1)^r} \ \ \ \textrm{ and } \ \ \ 
\Gamma_4 \ll \frac{(\log z)^r}{z } 
\frac{\prod_{i=1}^{r} x_i x'_i}
{W_z^{2r} (\log  x_1)^r}
,$$
where the implied constants depend only on $r$.
\end{lemma}\begin{proof} We start by making the 
  changes of variables 
$m_i=a_i+s_i W_z t_i$ and 
$m'_i=a'_i+s'_i W_z t'_i$  
so that the first two properties in $\Gamma_4$
imply   $|t_i| \leq 2M_i/W_z$ and 
$|t'_i| \leq 2M'_i/W_z.$ 
Let $n = 2r$ and identify $\mathbb{Z}^n$ with pairs 
$(\b t, \b t') \in  \Z^r \times \Z^r$. 
We shall use Lemma \ref{lem_lopsided} with $m=2$ and   
\[\Omega=\left\{ (\b t, \b t') 
\in \Z^r \times \Z^r : 
\prod_{i=1}^r (s_i a_i+W_z t_i, s'_i a'_i+W_z t'_i)_{\mathbb{Q}_p} = 1 \text{ for all primes } p>z \right\}.\] 
For any prime $p$, the local density $\omega(p)$ is defined by \[
\omega(p) = 1 - 
\frac{\# \Omega(\mathbb{Z}/p^2\mathbb{Z})}{p^{4r}}.\]
When $p\leq z$ then we clearly have $\omega(p)=0$. When $p>z$ we start 
estimating $\omega(p)$ by noting that $p\nmid W_z$
hence, an invertible change of variables in $\Z/p^2\Z$
shows that $$\# \Omega(\mathbb{Z}/p^2\mathbb{Z})= 
\#\left\{\b m,\b m'\in (\Z/p^2\Z)^{2r}:
\prod_{i=1}^r( m_i, m'_i )_{\Q_p}=1\right\}.$$
The contribution of the cases  with 
$p^2 \mid \prod_{i=1}^r m_i m'_i$ 
is at most 
$$2r p^{4r-2} + (2r)^2 (p^{2r-1})^2 \leq 5r^2p^{4r-2}$$
coming, respectively from cases where
$p^2$ divides one of the variables
or  when  $p$ divides at least two.  If $p\nmid \prod_i m_i m'_i$ and $p>2$
then every $(\b m,\b m') \in (\Z/p^2\Z)^{2r}$
falls into $\Omega(\Z/p^2\Z)$.
Hence,  \[  
-\frac{5r^2}{p^2}\leq \omega(p) -1 +\left(1-\frac{1}{p}\right)^{2r}+\lambda_p \leq 
 0
,\] where $\lambda_p=
 \#\{(\b m,\b m') \in \Omega(\mathbb{Z}/p^2\mathbb{Z}):
v_p(\prod_i m_i m'_i)=1\}p^{-4r}.$
For $p\neq 2$ 
and $v_p(\prod_i m_i m'_i)=1$
there exists a unique $i$ such that $p\mid m_im'_i$, in which case
\[ \prod_{i=1}^r (m_i,  m'_i)_{\mathbb{Q}_p}=\begin{cases} 
(\frac{m'_i}{p}), & \text{if } p \mid m_i,  \\
(\frac{m_i}{p}), & \text{if } p \mid m'_i.  \end{cases} \]
Since for $p>2$  there are $(p-1)/2$  squares in $\F_p^*$,
we obtain 
 \[\lambda_p= 
\left(1-\frac{1}{p}\right)^{2r-1}
\frac{r}{p} .\] Let us now note that for any real 
number $x\in [0,1]$ and $n\in \mathbb N$ we have 
$$ (1-x)^n \leq \mathrm e^{-nx} \leq 
1-nx +\frac{n^2 x^2}{2}
,$$ hence,  \begin{align*}
\omega(p) \geq  &1 - \frac{5r^2}{p^2}
-\left(1-\frac{2r}{p} +\frac{2r^2 }{p^2} \right) - 
\frac{r}{p}\left(1-\frac{2r-1}{p} 
+\frac{(2r-1)^2 x^2}{2} \right) 
\\ = &\frac{r}{p}-\frac{5r^2+r}{p^2}
- \frac{r(2r-1)^2 }{2p^3}  \geq 
\frac{r}{p} -\frac{8r^3}{p^2} \end{align*}
and 
$$\omega(p) \leq 
1 -\left(1-\frac{1}{p}\right)^{2r}-
\left(1-\frac{1}{p}\right)^{2r-1}
\frac{r}{p},$$  which shows that $ \omega(p) <1$ for all $p$ and that $\omega(p)=O_r(1/p)$, hence, 
$\limsup_{p\to\infty}\omega(p) =0$.
Thus,
$$L(t)= \sum_{q\leq t} \mu(q)^2 
\prod_{p\mid q }\frac{\omega(p)}{1-\omega(p)}
\geq 
\sum_{\substack{  q\leq t \\ 
p\mid q\Rightarrow p>z
}} \frac{\mu(q)^2 }{q} 
\prod_{p\mid q } r\big(1-\frac{8r^2}{p}\big)
\gg_r \Big(\frac{\log t}{\log z}\Big)^r$$
by \cite[Theorem 1.1]{MR2061214}. By the large 
sieve for lopsided boxes \cite[Lemma 5.1]{wilson}
we then obtain\[\Gamma_3 \ll 
\frac{\prod_{i=1}^{2r} B_i }
{(\log B_1)^r  }  ,\] 
where   $B_i:=2x_i/W_z$ for $i\leq r$
 and  $B'_i:=2x'_i/W_z$ for $i> r$. 

To bound $\Gamma_4$ we use Lemma \ref{lem_lopsided}
with  $Y$ given by the vanishing of 
$\prod_{i=1}^r m_i m'_i$. Then the lemma,
together with \eqref{eq_papoutsia},
show that  \[\Gamma_4 \ll \frac{(\log z)^r}{ z}
\frac{\prod_{i=1}^{2r} B_i }{(\log  B_1)^r } . \] 
Note that we used the  assumption 
$W_z \max_i \{x_i,x'_i\} \leq  
(\min \{x_i,x'_i\})^{1+\frac{1}{8n}}$
to verify the condition $\max_i B_i \leq (\min B_i)^
{1+\frac{1}{8(n-1)}}$
required for \eqref{eq_papoutsia}
and the fact that  $k\geq 2 $.\end{proof}
\begin{lemma} \label{lem_leo0}
Let $(X, \mathcal{F}, \mu)$ be a measure space, and let $E_2 \subset E_1 \subset X$ be measurable subsets.
Assume that $g, g_1: X \to [0,\infty)$ are integrable functions such that $g_1\rvert_{E_1} \leq g\rvert_{E_1}$.
Then\[\int_{E_1 \setminus E_2} g_1(x) \, \mathrm{d}\mu(x) \leq \int_{X \setminus E_2} g(x) \, \mathrm{d}\mu(x).
\]In particular, if $g_1\rvert_{E_2} = g\rvert_{E_2}$, then
\[0 \leq \int_{X} g(x) \, \mathrm{d}\mu(x) - \int_{E_1} g_1(x) \, \mathrm{d}\mu(x) \leq 
\int_{X \setminus E_2} g(x) \, \mathrm{d}\mu(x).
\]\end{lemma}\begin{proof}
Since $g_1\rvert_{E_1} \leq g\rvert_{E_1}$ and $E_1 \setminus E_2 \subset E_1$, we have
\[\int_{E_1 \setminus E_2} g_1(x) \, \mathrm{d}\mu(x) \leq \int_{E_1 \setminus E_2} g(x) \, \mathrm{d}\mu(x).
\]Because $g \geq 0$ and $E_1 \setminus E_2 \subset X \setminus E_2$, it follows that
\[\int_{E_1 \setminus E_2} g(x) \, \mathrm{d}\mu(x) \leq \int_{X \setminus E_2} g(x) \, \mathrm{d}\mu(x),
\]proving the first claim.

For the second claim, the assumptions
$g_1\rvert_{E_1} \leq g\rvert_{E_1}$ give
$\int_{E_1} g_1 \leq 
\int_{X} g$. Furthermore, using $g_1\rvert_{E_2} = 
g\rvert_{E_2}$ and $g_1 \ge 0$, we decompose the integrals over $X = E_2 \cup (E_1 \setminus E_2) \cup (X \setminus E_1)$:\[\int_{X} g - 
\int_{E_1} g_1 = \int_{E_1 \setminus E_2} (g - g_1) + \int_{X \setminus E_1} g \leq \int_{E_1 \setminus E_2} g + \int_{X\setminus E_1} g = 
\int_{X \setminus E_2} g.\qedhere\]
\end{proof}Now define 
\begin{equation}\label{del destin non vi lagnate} 
X=\left\{(\b m,\b m') \in 
\N ^r\times \N^r: 
\begin{array}{l} 
m_i \leq x_i, m_i\equiv a_i \md {W_z} \  \forall i
\\ m'_i \leq x'_i,  
m'_i\equiv a'_i \md {W_z} \  \forall i
\end{array}\right\}.\end{equation}Since $X$ is finite,
we can equip it with the structure 
of a probability space by assigning it the uniform 
discrete probability measure. Let us define the 
functions $g,g_1:X\to \{0,1\}$ by \begin{align}
g(\b m,\b m')&:=\prod_{p>z} \frac{1}{2}
\Big(1+\prod_{i=1}^r (s_i m_i,  s'_i m'_i)_{\mathbb{Q}_p}\Big),
 \label{girls_wanna_have_fun}
\\
g_1(\b m,\b m')&:=\prod_{i=1}^r\prod_{p>z} \frac{1}{2}
\Big(1+ (s_i m_i,  s'_i m'_i)_{\mathbb{Q}_p}\Big) 
.\end{align} Both functions are well-defined
as the products over $p>z$ are finite because 
the $p$-adic Hilbert symbols become trivial 
when $p\nmid \prod_{i} m_im'_i$ and $p>2$.
The function $g$ is the detector of the condition 
$$\prod_{i=1}^r (s_i m_i,s'_i m'_i)_{\mathbb{Q}_p}
=1 \ \ \ \forall p>z$$ and $g_1$ is the detector 
of the condition 
$$ (s_i m_i,  s'_i m'_i)_{\mathbb{Q}_p}=1 
\ \ \ \forall i\geq 1, \forall p>z.$$
As shown in Lemma \ref{lem_hilbert_symbol_bound}, 
the average of $g$ over $X$ is essentially equal 
to its average over   
$$  E:=\left\{(\b m,\b m') \in X:
p^2\nmid \prod_{i=1}^r m_im'_i \ \ 
\forall p>z\right\} .$$
Since $g_1$ is technically more tractable than $g$, 
it is desirable to have the freedom to 
replace averages of $g$ over $X$
by  averages of $g_1$ over $X$ or  subsets of $X$.
This transference is justified by the following result:
\begin{lemma}\label{lem:aurioprwistis7}
Let $r,z,m(z),W_z,\b x,\b x',\b a,
\b a'$ be as in Lemma \ref{lem_hilbert_symbol_bound}.
Then for any subset $E' \subset X$ we have 
 \begin{align*}
&\sum_{(\b m,\b m') \in  X} \ 
\prod_{p>z} \frac{1}{2}
\Big(1+\prod_{i=1}^r (s_i m_i,  s'_i m'_i)_{\mathbb{Q}_p}\Big)
\\=&\sum_{(\b m,\b m') \in  E\cup E'}\ 
\prod_{i=1}^r\prod_{p>z} \frac{1}{2}
\Big(1+ (s_i m_i,  s'_i m'_i)_{\mathbb{Q}_p}\Big)+
O\l(\frac{(\log z)^r}{z } 
\frac{\prod_{i=1}^{r} x_i x'_i}
{W_z^{2r} (\log x_1)^r}\r),\end{align*}
where the implied constant depends only on $r$.
\end{lemma}\begin{proof} If $g_1(\b m,\b m')=1$
then for all $i\geq 1 $ and all primes $p>z$ we have 
$ (s_i m_i, s'_i  m'_i)_{\mathbb{Q}_p}=1$, hence, 
$\prod_{i=1}^r (s_im_i, s'_i m'_i)_{\mathbb{Q}_p}=1$.
This shows that $g_1\leq g$ and it is easy to see
that these two conditions coincide when restricting 
$(\b m,\b m')$ on $E$. Thus, we are free to employ 
Lemma \ref{lem_leo0} with $E_2=E$ and $E_1=E\cup E'$,
which  gives 
\[\l|\int_{X} g - \int_{E\cup E'} g_1  \r|\leq 
\int_{X \setminus E} g.\] The two integrals on 
the left-hand side coincide with the 
corresponding sums in the present lemma, 
while $\int_{X \setminus E} g$ is   
the quantity $\Gamma$ in Lemma~\ref{lem_hilbert_symbol_bound}, 
which, once invoked, completes the argument.\end{proof}

\section{Character sums}\label{s_char_sums}
Assume that $\b a,\b a' \in (\Z/W_z\Z)^r$ where $W_z$ is as in \eqref{Porpora Torbido intorno al core}.
We shall estimate   asymptotically
$$ \c C(\b x, \b x'; W_z, 
\b a,  \b a'):=
\#\left\{\b m,\b m' \in \N^r: 
\begin{array}{l} 
m_i \leq x_i, m_i\equiv a_i \md {W_z} \  \forall i,
\\
m'_i \leq x'_i,  m'_i\equiv a'_i \md {W_z} \  \forall i,
\\ 
\prod_{i=1}^r (s_i m_i,s'_i m'_i )_{\Q_p}=1\ \ \forall 
\textrm{ prime } p 
\end{array}
\right\},
$$ where $ x_i,x'_i \geq 1$ and 
$s_i,s'_i \in \{-1,1\}$ satisfy $$
\prod_{i=1}^r (s_i,s'_i)_\R=1.$$
We make the assumption  
\begin{equation}\label{lem_bill_evans}
\max\{v_p(a_i),v_p(a'_i)\}\leq  m(z)-1-2
\mathds 1_{\{2\}}(p)
\ \ \ \forall p\leq z.\end{equation} 
This condition simplifies the Hilbert symbol 
condition at primes $p\leq z$ as the next lemma shows:
\begin{lemma}\label{lem_stabili}
If \eqref{lem_bill_evans} holds then for all 
$\b m,\b m'$  in $\c C(\b x, \b x'; W_z, 
\b a,  \b a')$, all primes $p\leq z$ and all 
$i\leq r $ we have $ (s_i m_i,s'_i m'_i )_{\Q_p}= 
(s_i a_i,s'_i a'_i )_{\Q_p}.$\end{lemma}
\begin{proof}
Assume $2<p\leq z$. 
Since $m_i\equiv a_i \md W$ and 
$v_p(a_i) < m=v_p(W)$, we deduce that $v_p(m_i)=v_p(a_i)$ and 
$m_i p^{-v_p(m_i)} \equiv a_i p^{-v_p(a_i)} \md p$.
Similarly we have $v_p(m'_i)=v_p(a'_i)$ and $m'_i p^{-v_p(m'_i)} 
\equiv a'_i p^{-v_p(a'_i)} \md p$. Hence,
\begin{align*}(s_i m_i,s'_i m'_i)_{\Q_p}&=
\Big(\frac{-1}{p}\Big)^{v_p(m_i) v_p( m'_i) } 
\Big(\frac{s_i m_i p^{-v_p(m_i)}}{p}\Big)^{v_p(m'_i)} 
\Big(\frac{s'_i m'_i p^{-v_p(m'_i)}}{p}\Big)^{v_p(m_i)}\\&=
\Big(\frac{-1}{p}\Big)^{v_p(a_i) v_p(a'_i) } 
\Big(\frac{s_ia_i p^{-v_p(a_i)}}{p}\Big)^{v_p(a'_i)} 
\Big(\frac{s'_i a'_i p^{-v_p(a'_i)}}{p}\Big)^{v_p(a_i)}=
(s_i a_i,s'_ia'_i)_{\Q_p}.
\end{align*} The proof works in a similar way for $p=2$. \end{proof} Next, we use the geometric-large sieve to
replace the counting function $\mathcal{C}$ 
with a product of simpler counting functions.
\begin{lemma}\label{talor_guerriero_invitto}
Keep the setting of 
Lemma \ref{lem_hilbert_symbol_bound},
assume \eqref{lem_bill_evans} and that 
$$\prod_{i=1}^r (s_i a_i,s'_i a'_i )_{\Q_p}=1 
\ \ \ \forall p\leq z.$$Then $$\c C(\b x, \b x'; W_z, 
\b a,  \b a')=\prod_{i=1}^r
\c N_{s_i,s'_i}(x_i,x'_i;W_z,a_i,a'_i)+
O\l(\frac{(\log z)^r}{z } 
\frac{\prod_{i=1}^{r} x_i x'_i}
{W_z^{2r} (\log x_1)^r}\r),$$ where 
$$\c N_{s,s'}(x,x';W_z,a,a'):=
\sum_{\substack{ 1\leq m\leq x,
1\leq m'\leq x'\\ (m,m')\equiv (a,a')\md {W_z} }}
\ \ \ \prod_{ \substack{ p\mid  m \\ p>z }}
\frac{1}{2} \left(1+\Big(\frac{ s' m' }{p}\Big)\right)
  \prod_{ \substack{ p\mid  m' \\ p>z }}
\frac{ 1}{2} \left(1+\Big(\frac{ s m }{p}\Big)\right)
$$ and the implied constant depends at most on $r$. 
\end{lemma}
\begin{proof}By \eqref{lem_bill_evans} and our assumption 
$\prod_{i=1}^r (s_i a_i,s'_i a'_i )_{\Q_p}=1 $ 
for $p\leq z$  we only need to check the Hilbert symbol 
condition in $\c C$ for primes $p>z$. In other words,
the assumptions guarantee that 
$$\c C(\b x, \b x'; W_z, 
\b a,  \b a')= 
\sum_{(\b m,\b m') \in  X} g(\b m,\b m'),$$
where $X$ and $g$ are defined respectively 
in \eqref{del destin non vi lagnate} and  
\eqref{girls_wanna_have_fun}. By 
Lemma \ref{lem:aurioprwistis7}
with $$ E'=\left\{(\b m,\b m') \in X:
p^2\nmid m_im'_i \ \ 
\forall p>z, \ \forall i\right\} $$
we deduce that 
$$\sum_{(\b m,\b m') \in  X} g(\b m,\b m')
=\sum_{(\b m,\b m') \in E'}\ 
\prod_{i=1}^r\prod_{p>z} \frac{1}{2}
\Big(1+ (s_i m_i, s'_i m'_i)_{\mathbb{Q}_p}\Big)+
O\l(\frac{(\log z)^r}{z } 
\frac{\prod_{i=1}^{r} x_i x'_i}
{W_z^{2r} (\log x_1)^r}\r).$$
Since $p>z$ ensures that $p\neq 2 $
and $p^2\nmid m_i m'_i$ we can write 
$$  (s_i m_i,s'_i n'_i )_{\Q_p}= 
\begin{cases} \ \ \  1 & \text{if } p\nmid 
m_i m'_i, \\ \Big(\frac{ s'_i m'_i }{p}\Big) & 
\text{if } p\mid m_i,\\
\Big(\frac{ s_i m_j }{p}\Big) & 
\text{if } p\mid m'_i.\end{cases} $$ 
Hence, up to an acceptable error term, 
 $ \c C(\b x, \b x'; W_z, \b a,  
 \b a')$ equals
 $$  \sum_{ \b m,\b m' \in \N^r }
 \prod_{i=1}^r \prod_{ \substack{ p\mid  m_i \\ p>z }}
\frac{ \left(1+\Big(\frac{ s'_i m'_i }{p}\Big)\right)}{2}
  \prod_{ \substack{ p\mid  m'_i \\ p>z }}
\frac{  \left(1+\Big(\frac{ s_i m_i }{p}\Big)\right)}{2},$$
where the sum is subject to 
the following conditions for all $i=1,\ldots, r$:
$$1\leq  m_i \leq x_i, 1\leq
m'_i \leq x'_i, \ \ \  (m_i,m'_i)
 \equiv (a_i,a'_i) \md {W_z}, \ \ \ 
 p^2\nmid  m_i m'_i \  \ \forall p>z.$$
The final sum can then be factored as
a product of independent sums, as stated in the lemma.
\end{proof} 
As a preparation for the character sum method
we first factor $m,m'$ in $\c N$
into a friable and a non-friable part.
\begin{lemma} \label{le_dove e leon?oe} 
Let $W_z$ be as in Lemma \ref{lem_hilbert_symbol_bound} and assume that $a,a'\in \Z/W_z\Z$ 
are such that  for all primes 
$p\leq z$ the property 
$\max\{v_p(a),v_p(a')\}\leq  m(z)-1-2
\mathds 1_{\{2\}}(p)$ holds. Define 
$$ w : = \prod_{p\leq z } p^{v_p(a)} 
\ \ \ \textrm{ and } 
\ \ \  w' : = \prod_{p\leq z } p^{v_p(a')}.$$
Then $\c N_{s,s'}(x,x';W_z,a,a')$ equals 
$$
\sum_{\substack{ 1\leq f \leq x/w,
1\leq f' \leq x'/w', \ 
\\ f\equiv a/w\md {W_z/w}
\\ f'\equiv a'/w'\md {W_z/w'}}} \mu(ff')^2
\ \ \ \prod_{p\mid f} 
\frac{1}{2} \left(1+\Big(\frac{s'w'f'}{p}\Big)\right)
\prod_{  p\mid  f' }
\frac{ 1}{2} \left(1+\Big(\frac{swf}{p}\Big)\right)
.$$\end{lemma}\begin{proof} 
Since $(m,m')\equiv (a,a')\md {W_z}$ we infer that 
$v_p(m)=v_p(w)$ and 
$v_p(m')=v_p(w')$ for all primes $p\leq z$.
Therefore, $ w \mid m$ and $ w' \mid m'$,  
 hence, both  $f= m/w,  f'=m'/w'$  are integers.
 Note that the condition $p^2\nmid  m m'$ for all primes 
 $p>z$ is equivalent to $f f'$ being square-free.
Furthermore, the congruence 
$m \equiv a \md {W_z}$ is the same as 
$f \equiv a/w \md{W_z/w}$. 
Since $W_z/w$ is divisible by every prime $p \le z$ 
(due to our assumed upper bound on $v_p(a)$), 
it is automatic from 
$f \equiv a/w \md{W_z/w}$ 
that $f$ is coprime to all primes 
$p \le z$.\end{proof}

\begin{lemma} \label{le_omega}
Keep the setting of Lemma \ref{le_dove e leon?oe}
and assume that  $\max\{x,x'\} \leq \min\{x,x'\}^2$ and that 
  $W_z\leq \log (3x)$.
  Then  the sum over $f,f'$ in Lemma \ref{le_dove e leon?oe} 
  equals 
 $$(1+\epsilon)
 \c S+ O\Big(\frac{xx'}{(\log (3x))^{1926} } \Big),$$  where 
 $\epsilon:=(s,s')_{\R}\prod_{p\leq z}
(sa,s'a')_{\Q_p}$, $$\c S:=  \sum_{\substack{ \ell \leq x/w,
 n \leq x'/w', \ 
\\ \ell\equiv a/w\md {W_z/w}
\\ n\equiv a'/w'\md {W_z/w'}}}  
\frac{\mu(\ell n)^2}{\tau(\ell)\tau(n)} 
$$ and  the implied constant is absolute.
 \end{lemma} \begin{proof} Opening up the products in 
 the sum over $f,f'$ in Lemma \ref{le_dove e leon?oe}
 we can write  $$
 \prod_{p\mid f} 
 \left(1+\Big(\frac{s'w'f'}{p}\Big)\right)=
 \sum_{\substack{k,\ell \in \N  \\ k\ell=f} }\Big(\frac{s'w'f'}{k}\Big), \ \ \ 
\prod_{  p\mid  f' }
 \left(1+\Big(\frac{swf}{p}\Big)\right)
= \sum_{\substack{m,n \in \N  \\ mn=f'}} \Big(\frac{swf}{m}\Big). $$ The terms with $k=m=1$ come from choosing $1$
in each bracket in both products over $p\mid f $ and 
$p\mid f'$. These terms contribute a quantity that equals $\c S$. 
Similarly, the terms 
with $\ell=n=1$ come from choosing the quadratic 
symbol $(\frac{\cdot}{p})$ in
each product over $p\mid f$ and $p\mid f'$. 
These terms therefore contribute  
$$  \sum_{\substack{ k \leq x/w, \  m \leq x'/w', \ 
\\ k\equiv a/w\md {W_z/w}
\\ m\equiv a'/w'\md {W_z/w'}}} 
\frac{\mu(km)^2}{\tau(k)\tau(m)}
\Big(\frac{s'w'm}{k}\Big)
 \Big(\frac{swk}{m}\Big)
.$$ The product of the two quadratic symbols is independent of $k,m$. Indeed, 
$$ \Big(\frac{s'w'm}{k}\Big)
 \Big(\frac{swk}{m}\Big)= \prod_{p\mid k}
 \Big(\frac{s'w'm}{p}\Big)\prod_{p\mid m}
  \Big(\frac{swk}{p}\Big).$$ By an argument 
  that is similar to the one in Lemma \ref{lem_stabili}
  this can be seen to be $$\prod_{p\mid km }
  (swk,s'w'm)_{\Q_p}=\prod_{p>z}
  (swk, s'w'm)_{\Q_p}=(s,s')_{\R}\prod_{p\leq z}
  (swk,s'w'm)_{\Q_p},$$ where the last equality 
  comes from Hilbert's product formula.
  Since $(wk, w'm)\equiv (a,a') \md{W_z}$ we infer that 
  \begin{equation}\label{zafferano} \Big(\frac{s'w'm}{k}\Big)
 \Big(\frac{swk}{m}\Big)=(s,s')_{\R}\prod_{p\leq z}
  (sa,s'a')_{\Q_p}=\epsilon.\end{equation}
  Hence, the contribution of the terms 
with $\ell=n=1$ is $\epsilon\c S$.

We will show that the remaining terms go into the error term. Let $\Psi=(\log (3x) )^A$ where $A>0$ is a parameter 
 that will be chosen later. The terms with $k,\ell \leq \Psi$ 
 contribute a quantity whose modulus is at most $\#\{
f\leq \Psi^2, f'\leq x'\} \leq \Psi^2 x'  \ll
xx'\Psi^{-1/6}$ due to the assumption
$\max\{x,x'\} \leq \min\{x,x'\}^2$. Similarly, the terms with 
$m,n \leq \Psi$ contribute at most $ \ll xx'\Psi^{-1/6} $. 

Next, we see that the terms with $k,n\leq \Psi$ 
contribute $$ \ll \sum_{k,n  \leq \Psi  } 
 \left| \sum_{\substack{ 
 \ell,m \in \N, \gcd( \ell m, kn )=1\\  
 \ell  \leq x/(kw), 
\ell\equiv  a/(kw) \md {W_z/w}\\ 
m \leq x'/(nw'),  
m   \equiv a'/(nw') \md {W_z/w'} } } 
\Big( \frac{ \mu( \ell  )^2 }{\tau(  \ell  )}  \Big)\cdot  
\Big( \frac{  \mu(  m  )^2}{\tau( m )}    \Big( \frac{  m }{k }   \Big) \Big(\frac{sw k   }{m} \Big)\Big) \cdot 
  \Big(\frac{ \ell  }{m} \Big)\right|.$$
  Noting that the congruence conditions on $\ell$ and $m$ 
  force them to be odd, we can use  the large  sieve for quadratic characters \cite[Lemma 2]{MR2675875} to bound this by 
$$\ll \sum_{k,n  \leq \Psi  } 
\Big( 
\frac{x}{k}
\Big(\frac{x'}{n}\Big)^{5/6}
+
\Big(\frac{x}{k}\Big)^{5/6}
\frac{x'}{n}
\Big) (\log (3 x ))^{7/6}\ll 
\frac{xx'\Psi^{1/6 }  (\log (3 x))^{13/6}}{\min\{x,x'\}^{1/5}}
\ll \frac{xx' (\log (3 x ))^{13/6}}{\Psi^{1/6} }$$ due to the 
assumption $\max\{x,x'\} \leq \min\{x,x'\}^2$.
One can  prove the same bound in a similar way for the contribution 
of the range $\ell,m \leq \Psi$.  
The range $k,n>\Psi$ contributes 
$$\ll \sum_{\substack{ \ell   \leq x/(\Psi w) \\ 
m  \leq x'/(\Psi w') } } 
 \left| \sum_{\substack{
 k,  n  \in \N^2, \gcd(  kn,   \ell m )=1\\ 
 \Psi<k  \leq x/(\ell w), 
 k\equiv a/(\ell w) \md {W_z/w}
 \\ \Psi<n \leq x'/(m w'),  
n  \equiv \alpha'/(mw') \md {W_z/w'} } } 
\Big( \frac{ \mu(k)^2}{\tau(k)}
 \Big(\frac{ s' w' m   }{k }  \Big)  \Big(\frac{  k  }{m} \Big)\Big)
\cdot \Big( \frac{ \mu(n)^2}{\tau(n)}  \Big) 
\cdot  \Big(\frac{ n }{k }  \Big)
 \right| $$ and by \cite[Lemma 2]{MR2675875} this is     
 $$\ll \sum_{\substack{ \ell   \leq x/\Psi \\ 
m  \leq x'/\Psi } }   \Big(  \frac{x}{\ell}
\Big(\frac{x'}{m}\Big)^{5/6}
+\Big(\frac{x}{\ell}\Big)^{5/6}
\frac{x'}{m} \Big) (\log (3 x ))^{7/6}\ll \frac{xx'} {\Psi^{1/6} }
(\log (3 x ))^{13/6}.$$ The cases with $\ell,m>\Psi$ are  
treated analogously.

The only remaining cases satisfy    $\max\{k,m\}\leq \Psi$ 
or $\max\{\ell,n\}\leq \Psi$. The terms with  $m\neq 1$ and
$\max\{k,m\}\leq \Psi$  contribute at most  $$
 \sum_{\substack{k,m,n  \in \N, \gcd(k, W_z)=1 \\ 
 k, m   \leq \Psi,  mn \leq x'/w',  \\ 
 m\geq 2,  mn  \equiv a'/w' \md {W_z/w'} } } 
\frac{ \mu(k  mn )^2}{\tau(kmn )}  
 \left| \sum_{\substack{  \ell  \leq x/ (kw), 
\gcd(  \ell, k n )=1 \\   
 \ell\equiv  a/(kw) \md {W_z/w}  } } 
\frac{ \mu( \ell  )^2}{\tau( \ell )} 
 \Big(\frac{  \ell  }{m}  \Big)\right|.$$   
 To bound the the sum over $\ell$ we employ \cite[Corollary 2]
 {MR2675875}  with $\chi_2(\cdot)=(\frac{\cdot}{m})$, 
 $q_2=m$, $q_1=W_z/w$ and $d$ being the product of all primes in 
 $kn$ that do not divide $m W_z$. We note that the assumptions in 
 our lemma show that $\ell $ is coprime to $W_z$ when 
 $\ell \equiv  a/(kw) \md{W_z/w}$ and the presence of the 
 quadratic symbol in the sum over $\ell$ ensures 
 that $\ell$ is coprime to $m$. Since   $\chi_2$   
 non-principal due to $m\geq 2 $,
we   obtain    the following bound for every fixed $C>0$: 
$$\ll  \sum_{\substack{  k, m   \leq \Psi, 
n \leq x'/m   } } 
\frac{  \mu(kmn )^2   } {\tau(kmn )}   \tau(kn)
\frac{ W_z m x}{k (\log x)^{C}  } \ll
\frac{W_z \Psi  x}{(\log x)^{C} }
\sum_{\substack{  k, m   \leq \Psi\\ 
n \leq x'/m  } } 
\frac{ 1 }{k  }
\ll \frac{W_z \Psi  xx'}{(\log 3x)^{C} }
(\log 3\Psi )^2
.$$The same bound can   be proved  for the contribution of the cases 
 with $\max\{k,m\}\leq \Psi$ and $k\geq 2 $ by applying 
 \cite[Corollary 2]{MR2675875} to the sum over $n$.

The cases with  $\max\{\ell,n\}\leq \Psi$ 
have a slightly different treatment due to a reciprocity factor.
They contribute  $$
\sum_{\substack{ 
\max\{\ell,n\}\leq \Psi \\ \gcd(\ell n,W_z )=1 }}
\frac{\mu(\ell n )^2}{\tau(\ell  n)} 
\sum_{\substack{k \leq x/(\ell w), \ 
m \leq x'/( nw'),  \  \gcd(km,\ell n)=1,
\\ k \equiv a/(\ell w) \md {W_z/w}
\\ m\equiv a'/(nw')\md {W_z/w'}}} 
\frac{\mu(k m )^2}{\tau(k m  )}
\left[\Big(\frac{s'w'm }{k}\Big)
\Big(\frac{s w k }{m}\Big) \right]
\Big(\frac{n }{k}\Big)
\Big(\frac{\ell }{m}\Big) 
$$ and we now observe that the quantity in the brackets $\l[\cdot\r]$
is independent of the summation variables $k,m$ due to 
\eqref{zafferano}. Once it is moved out of the inner sum
we can bound the contribution of the terms with 
$\ell \geq 2 $ by applying 
 \cite[Corollary 2]{MR2675875} to the sum over $m$,
thus, obtaining the overall bound  
 $$ \ll  \frac{ W_z    \Psi^2  x x' }{(\log x')^C} 
 \ll  \frac{ W_z    \Psi^2  x x' }{(\log x)^C} 
,$$ where $C$ is an arbitary fixed constant.
The cases with 
$n\geq 2 $ and  $\max\{\ell,n\}\leq \Psi$ 
can be treated similarly, leading to a bound of the same 
quality.  Finally, 
gathering all error terms we end up with the overall 
contribution $$ \ll 
xx'\l( \frac{ (\log (3 x ))^{13/6}}{\Psi^{1/6} }
+ \frac{ W_z    \Psi^2   }{(\log x)^C} \r).$$ Set $A := C/4$ 
and recall that $\Psi = (\log(3x))^A$ and $W_z \le \log(3x)$. 
Redefining the constant $C$ completes the proof.\end{proof}
Let 
$$ c_0:= \frac{1}{\sqrt \pi } \prod_{p=2}^\infty
\Big(1+\frac{1}{2p} \Big) \Big(1-\frac{1}{p} \Big)^{1/2} $$
and define the mutiplicative function $c:\mathbb N \to [0,1]$ by 
$$ c(r):= \prod_{p\mid r }\Big(1+\frac{1}{2p} \Big)^{-1}. $$

\begin{lemma} \label{Fasch oboe}Keep the setting of 
Lemma \ref{le_omega} and let $\delta \in \mathbb N$ with $\delta \leq (\log  3x)^9$ that is coprime to $W_z$. Then 
$$ \sum_{\substack{ \ell \leq x/(w\delta), \gcd(\ell,\delta)=1\\
\ell \equiv a/(w\delta)\md{W_z/w} }} 
\frac{\mu(\ell)^2}{\tau(\ell)} = 
\frac{ c_0 c(\delta )c( W_z)}{\delta \phi(W_z)} 
\frac{x}{\sqrt{\log (3x)}}
\l\{1+O\l(\frac{(\log \log (3 x))}{\log (3x)}\r)\r\}, $$ 
where the implied constant is absolute.
\end{lemma}\begin{proof}Applying \cite[Corollary 2]{MR2675875} 
with $d=\delta$, $q_1=W_z/w$ and $q_2=1$ gives the asymptotic 
$$\frac{ c_0 c(\delta W_z/w)}{\phi(W_z/w)} 
\frac{x/(w\delta)}{\sqrt{\log (x/(w\delta))}}
\l\{1+O\l(\frac{(\log \log (3 \delta W_z))^{3/2}}{\log (x/(w\delta))}\r)\r\} +O_C\l(
\frac{\tau(\delta) W_zx/(w\delta)}{(\log (x/(w\delta)))^C} \r)
$$ for any fixed $C>0$. Since $w\leq W_z \leq \log(3x)$ and 
$\delta \leq (\log (3x))^9$ we infer that 
$$
\frac{1}{\sqrt{\log (x/(w\delta ))}}
=\frac{1}{\sqrt{\log  x}} \l(1+O\l(\frac{\log \log (3x)}{\log x }\r)\r)
$$ so that the asymptotic becomes 
$$\frac{ c_0 c(\delta )c( W_z/w)}{\delta W_z} 
\frac{W_z/w}{\phi(W_z/w)} 
\frac{x}{\sqrt{\log (3x)}}
\l\{1+O\l(\frac{(\log \log (3 x))}{\log (3x)}\r)\r\} +O_C\l(  
\frac{x}{(\log (3x))^{C-1}} \r)
.$$ The proof concludes by taking $C$ large enough and
noting that  the radicals of 
$W_z/w$  and $W_z$ agree, hence, 
$c(W_z/w)=c(W_z)$ and $(W_z/z) \phi(W_z/w)^{-1}= W_z\phi(W_z)^{-1}$. 
\end{proof}

\begin{lemma} \label{cracking} Keep the setting of 
Lemma \ref{le_omega}
and assume that $z\leq (\log (3x))^{1/2}$. Then 
$$ \c S= 
\frac{1}{ \pi W_z\phi(W_z)}
\frac{xx'}{\sqrt{(\log x)(\log x')}}
(1+O(z^{-1}))
,$$ where the implied constant is absolute.
\end{lemma}
\begin{proof} We start by using M\"obius inversion to write
$$\c S=
\sum_{\substack{\delta \leq xx' \\ \gcd(\delta,W_z)=1 }}
\frac{\mu(\delta)}{\tau(\delta)^2}
\l(\sum_{\substack{ \ell \leq x/(w\delta), \ \gcd(\ell,\delta)=1,\\
\ell \equiv a/(w\delta)\md{W_z/w} }} 
\frac{\mu(\ell)^2}{\tau(\ell)}\r)
\l(\sum_{\substack{ n \leq x'/(w'\delta), \ \gcd(n,\delta)=1,\\
n \equiv a'/(w'\delta)\md{W_z/w'} }} 
\frac{\mu(n)^2}{\tau(n)}\r)
.$$ The contribution of the terms with $\delta>(\log 3x)^9$
has modulus 
$$\ll 
\sum_{\delta>(\log 3x)^9 }
\frac{xx'}{\delta^2}\ll \frac{xx'}{(\log 3x)^9}
.$$ For the remaining $\delta$ we apply 
Lemma \ref{Fasch oboe} twice  to obtain 
$$\frac{c_0^2c( W_z)^2}{\phi(W_z)^2}
\frac{xx'}{\sqrt{(\log x)(\log x')}}
\l\{1+O\l(\frac{(\log \log (3 x))}{\log (3x)}\r)\r\}
\sum_{\substack{\delta \leq (\log 3x)^9  \\ \gcd(\delta,W_z)=1 }}
\frac{\mu(\delta)}{\tau(\delta)^2} 
\frac{ c(\delta )^2}{\delta^2 }
$$ due to our assumptions that ensure that 
$\log x \ll \log x' \ll \log x$.
Since $|c(\delta)|\leq 1 $ we get 
$$\sum_{\substack{\delta \leq (\log 3x)^9  \\ \gcd(\delta,W_z)=1 }}
\frac{\mu(\delta)}{\tau(\delta)^2} 
\frac{ c(\delta )^2}{\delta^2 }= \prod_{p>z}
\l(1-\frac{1}{(2p+1)^2}\r)+O((\log (3x))^{-9})=1
+O(z^{-1})$$ due to the assumption $z\leq (\log (3x))^{1/2}$.
The leading constant is
\[
\frac{c_0^2c( W_z)^2}{\phi(W_z)^2}
=
\frac{c( W_z)^2}{\phi(W_z)^2}
\frac{(1+O(z^{-1}))}{ \pi } 
\prod_{p\leq z}
\Big(1+\frac{1}{2p} \Big)^2 \Big(1-\frac{1}{p} \Big)=
\frac{1}{W_z\phi(W_z)}
\frac{(1+O(z^{-1}))}{ \pi } 
.\qedhere\]
\end{proof}
Injecting Lemma \ref{cracking} into Lemma \ref{le_omega} yields the
following:\begin{lemma}\label{faesk}Keep the setting of 
Lemma \ref{cracking}.
Then    
 $$\c N_{s,s'}(x,x';W_z,a,a')=
 \frac{(1+\epsilon)}{ \pi W_z\phi(W_z)}
\frac{xx'}{\sqrt{(\log x)(\log x')}}
(1+O(z^{-1}))
+ O\Big(\frac{xx'}{(\log (3x))^{1926} } \Big)
,$$ where $\epsilon$ is defined in the statement of Lemma \ref{le_omega} and  the implied constant is absolute.
\end{lemma}

\begin{proposition}\label{prop_main}Let $r,z\geq 1 $ and 
$W_z$ be
as in \eqref{Porpora Torbido intorno al core}.
Assume that we are given 
$\b a,\b a' \in (\Z/W_z\Z)^r$ that satisfy   $$ 
\max\{v_p(a_i),v_p(a'_i)\}\leq  m(z)-1-2
\mathds 1_{\{2\}}(p)
\ \ \ \forall p\leq z.$$Let 
$\b s, \b s' \in  \{-1,1\}^r$ be such that $
\prod_{i=1}^r (s_i,s'_i)_\R=1$ and 
$$\prod_{i=1}^r (s_i a_i,s'_i a'_i )_{\Q_p}=1 
\ \ \ \forall p\leq z.$$
Assume that   $\b x ,\b x'\in [1,\infty)^r$ 
fulfill 
$W_z \max_i \{x_i,x'_i\} \leq  
(\min \{x_i,x'_i\})^{1+\frac{1}{16r}}$  and   
  $W_z\leq \log (3\min _i x_i)$. 
  Then $$\c C(\b x, \b x'; W_z, 
\b a,  \b a')=
\frac{1}{(\pi W_z\phi(W_z))^r}
\prod_{i=1}^r
\frac{(1+\epsilon_i)x_ix'_i}{\sqrt{(\log x_i)(\log x'_i)}}
+
O\l(  \frac{(\log z)^r}{z } 
\frac{\prod_{i=1}^{r} x_i x'_i}
{W_z^{2r} (\log x_1)^r}\r)$$ where $ \epsilon_i:=
(s_i ,s'_i)_{\R}
\prod_{p\leq z}
(s_ia_i,s'_ia'_i)_{\Q_p}
$ and 
the implied constant depends only on $r$. \end{proposition}
\begin{proof}The assumption
$W_z\leq \log (3\min _i x_i)$ and the inequality 
$2^z \leq W_z$ combine to show 
that $z\ll \log \log (3\min _i x_i)$, hence, 
  $z\leq (\log (3x))^{1/2}$ when all $x_i$ are sufficiently large.
  Hence, we can     employ Lemma \ref{faesk}, 
  which can be combined with Lemma \ref{talor_guerriero_invitto}
  to   write $\c C(\b x, \b x'; W_z, 
\b a,  \b a')$ as $$\frac{(1+O(z^{-1}))}{(\pi W_z\phi(W_z))^r}
\prod_{i=1}^r
\frac{(1+\epsilon_i)x_ix'_i}{\sqrt{(\log x_i)(\log x'_i)}}
+O\l(\l(
\frac{(W_z/\phi(W_z))^{r-2} W_z^2}{(\log (3x_1))^{1925}}
+ \frac{(\log z)^r}{z } \r)  \frac{\prod_{i=1}^{r} x_i x'_i}
{W_z^{2r} (\log x_1)^r}\r).$$ 
Using the bound   $W_z\leq \log (3\min _i x_i)$ and 
the estimate $W_z/\phi(W_z) \ll \log z$ we can write this as 
\[\frac{1}{(\pi W_z\phi(W_z))^r}
\prod_{i=1}^r\frac{(1+\epsilon_i)x_ix'_i}
{\sqrt{(\log x_i)(\log x'_i)}}
+O\l(  \frac{(\log z)^r}{z } 
\frac{\prod_{i=1}^{r} x_i x'_i}
{W_z^{2r} (\log x_1)^r}\r).\qedhere\] 
\end{proof} 

\section{Circle method}\label{circle_method}
We recall a special case of 
\cite[Theorem 2.4]{11388}, which 
allows us to reduce 
Theorem \ref{thm_main} into an average 
of an arithmetic function over  progressions.  
Let $R$ be a strictly positive integer  and let 
$k:\mathbb Z^R \to \mathbb C$ be an arbitrary arithmetic function.
Assume that we are given 
functions  
$\omega:[1,\infty)^R\to \mathbb C$  of class $\c C^1$  and 
for any  $q\in \N$  and $\b a\in (\Z/q\Z)^R$ let  
$\rho(\b a,q)$ be any   element in $ \mathbb C$ satisfying
$|\rho(\b a,q)|\leq q$.
Define \begin{equation}\label{def_siegelwalf}
E(x;q):=\sup_{x_1,\ldots,  x_R \in \R\cap  [1,  x]} \ \ 
\max_{\substack{ \b a\in (\Z/q\Z)^R  \\ 
\gcd(\b a , q )=1 }}
\ \ \left| \sum_{\substack{ \b m  
\in \prod_{i=1}^R (\N\cap [1,x_i]) 
\\ \b m \equiv \b a \md q  }} 
k(\b m) -\rho(\b a,q)   
\int_{   \prod_{i=1}^R [1,x_i ]  }  \omega(\b t) \mathrm d \b t  \right| .
\end{equation} For $z\geq 2$ and any function  
$m(z):\Z\cap [1,\infty)\to \Z\cap [1,\infty)$ we let  \begin{equation}\label
{Porpora Torbido intorno al core}
W_z:=\prod_{p\leq z} p^{m(z)} .\end{equation}
Denote $ b=2 \max_i |F_i([-1,1]^n )| $.  
\begin{lemma}
[Destagnol--Lyczak--Sofos]
\label{lem:vachms} Let   $d,R,n$  be positive 
integers, $k:\mathbb Z^R\to \mathbb C$ 
be an arithmetic function, $z\geq 2 $,  
$s_1,\ldots, s_R \in \{-1,1\}$
and polynomials 
 $F_1,\ldots,F_R \in \Z[x_1,\ldots, x_n]$ be as in 
 Definition \ref{stronghardylittlew}. 
 Then for   all   
 $P \geq 1$  we have   \begin{align*}
\frac{1}{P^{n}} 
&\sum_{\substack{ \b t \in \Z^{n}\cap P[-1,1]^n \\ 
\min_j s_j F_j(\b t ) > 0 }}  
k(s_1 F_1(\b t), \ldots, s_R F_R(\b t) ) 
\\ &=    \left(\,\,  
\int\limits_{\substack{ \b t \in [-1,1]^n \\ 
 s_j F_j(\b t ) > P^{-d} \ \forall j} }  
\omega\left( P^d (s_j  F_j(\b t ))\right)  
\mathrm d \b t  \right)
 \sum_{\b t \in (\Z/W_z\Z )^{n} } 
 \frac{ \rho((s_j  F_j(\b t ) ) ,W_z) }{ W_z^{n-R} }
\\ &+O\left( \frac{\|k \|_1}{P^{Rd}  }  
\left(P^{-c}+z^{-c}
+\sum_{p\leq z} \frac{1}{p^{1+m(z)}} 
\right)  +\frac{E(bP^d; W_z ) W_z^R  }{P^{Rd}   } \right),
\end{align*} where the implied constant depends at most on 
$\b F$. The quantity  $c=c(\b F )>0$ is an explicit 
positive constant and $$ \|k \|_1:= \sum_{\boldsymbol \nu \in \N^R \cap [1 , b P^d]^R} 
|k(\boldsymbol \nu)|  .$$  \end{lemma}

To apply this, we  first write the function $N_\alpha(X)$ in Theorem \ref{thm_main} 
as $$N_\alpha(X)=
\sum_{\substack{\b t \in \Z^n \cap X[-1,1]^n \\ 
\prod_{i=1}^r f_i (\b t) g_i(\b t ) \neq 0 }}
\frac{1}{2} \Big(1+ \prod_{i=1}^r (f_i(\b t ),g_i(\b t ) )_{\R} \Big)
\prod_{p=2}^\infty \frac{1}{2} \Big(
1+ \prod_{i=1}^r (f_i(\b t ),g_i(\b t ) )_{\Q_p} \Big),$$  
where $(\cdot,\cdot)_{L}$ is the Hilbert symbol 
over a field $L$. 
Define for $\b s , \b s'\in \{-1,1\}^r$ the function
$ k_{\b s, \b s'} : \N^r \times \N ^r \to \{0,1\}$     by 
$$ k_{\b s, \b s'} ( \b n , \b n' )= \prod_{p=2}^\infty  \frac{1}{2} 
\Big(1+ \prod_{i=1}^r (s_i n_i,s'_i n'_i )_{\Q_p} \Big).$$
Fixing the signs we   obtain  \begin{equation}
\label{eq_ser} N_\alpha(X)=
\sum_{\substack{ \b s,\b s'\in \{-1,1\}^r\\ 
\prod_{i=1}^r (s_i,s'_i)_\R=1 }} \c C_{\b s ,\b s'}(X), 
\end{equation}
where $$ 
\c C_{\b s, \b s'}(X)=
\sum_{\substack{\b t \in \Z^n \cap X[-1,1]^n \\ 
s_i f_i (\b t) >0 \ \forall i \\
s'_i  g_i(\b t ) >0 \ \forall i }}
k_{\b s, \b s'}((s_1 f_1(\b t ), \ldots, s_r f_r(\b t ) ),
(s'_1 g_1(\b t ), \ldots, s'_r g_r(\b t ) ) ).$$ We estimate 
$\c C_{\b s, \b s'}$ by applying 
Lemma \ref{lem:vachms} with $k=k_{\b s, \b s'}$,
$R=2r$, $z=[(\log X)/(2 \log \log X)]$,
$m=[\log z]$ and  
\[F_i= \begin{cases}f_i, & \text{if } i=1,\ldots,r,  \\
g_i,  & \text{if } i=r+1,\ldots,2r.  \end{cases}\]
Define  \[\omega(\b t,\b t'):= \pi^{-r}
\begin{cases}\prod_{i=1}^r((\log t_i)(\log t'_i))^{-1/2}, 
& \text{if } 
\max_i \{t_i,t'_i\} 
\log (3\min_i t_i)
\leq  
(\min \{t_i,t'_i\})^{1+\frac{1}{16r}},  \\
0,  & \text{otherwise}  \end{cases}\]
and 
 \[\rho(\b a,\b a',W_z):= 
\begin{cases}
( W_z\phi(W_z))^{-r} \prod_{i=1}^r(1+\epsilon_i) , 
& \text{if } 
\eqref{eq:clre} \text{ holds},  \\
0,  & \text{otherwise,}  \end{cases}\]
where $\epsilon_i$ is defined in Proposition \ref{prop_main}
and  \begin{equation}\label{eq:clre}
\begin{cases}
\max\{v_p(a_i), v_p(a'_i)\} \leq m(z) - 1 - 2\mathds{1}_{\{2\}}(p), & \forall p \leq z, \forall i \\[1ex]
\displaystyle\prod_{i=1}^r (s_i a_i, s'_i a'_i)_{\mathbb{Q}_p} = 1, & \forall p \leq z.  
\end{cases}
\end{equation}  Then, from Lemma \ref{lem:vachms} there exists 
a strictly positive constant 
 $c=c(\b f,\b g )$ such that  
\begin{equation}
\label{eq_dento8imi8iamiaxarapsixoula} 
\frac{\c C_{\b s,\b s'}(X)}{X^{n}}=  I_X  S_z
+O\left( \frac{\|k_{\b s,\b s'} \|_1}{X^{2rd}  }  
\left(X^{-c}+z^{-c}
+\sum_{p\leq z} \frac{1}{p^{1+m(z)}} 
\right)  +\frac{E(bX^d; W_z ) W_z^{2r}  }{X^{2rd}   } \right),
\end{equation}
with
$$ I_X :=   
\int\limits_{ \substack{  \b t \in [-1,1]^n 
\\ \min_i s_i f_i(\b t )>X^{-d}
\\ \min_i s'_i g_i(\b t )>X^{-d}
}}  
\omega(X^d ( s_i  f_i(\b t ) )_{i=1}^r, 
X^d(s'_i  g_i(\b t ) )_{i=1}^r)
\mathrm d \b t $$ and 
\begin{equation}
\label{vivaldi-griselda} 
S_z:= \frac{  W_z^{r-n} }
 {  \phi(W_z)^{r} } \hspace{-0.5cm}
 \sum_{\substack{ \b t \in (\Z/W_z\Z )^{n}, \ 
  \prod_{i=1}^r (f_i(\b t ), g_i(\b t ))_{\mathbb{Q}_p} = 1, 
\forall p \leq z \\
 \max\{v_p(f_i(\b t)), v_p(g_i(\b t))\} \leq 
 m(z) - 1 - 2\mathds{1}_{\{2\}}(p), \forall p \leq z, \forall i
} } 
\prod_{i=1}^r \l(1+
(s_i ,s'_i)_{\R}
\prod_{p\leq z}
(f_i(\b t),g_i(\b t))_{\Q_p}\r)
.\end{equation}
\begin{lemma}\label{vivaldi RV 580}
We have $E(bX^d; W_z ) \ll  (\log z)^r z^{-1}
X^{2dr}(W_z^2\log X)^{-r}$ with an implied constant that is 
independent of $z$ and $X$.
\end{lemma}\begin{proof} Recalling \eqref{def_siegelwalf} 
we employ  Proposition \ref{prop_main}  to bound 
  the contribution of those $x_i$ in  
\eqref{def_siegelwalf} that satisfy 
  $X^d/(\log X)^{2r}<x_i \leq b X^d$ for all $i$.
This gives 
$$
 \frac{(\log z)^r}{z } 
\frac{\prod_{i=1}^{r} x_i x'_i}
{W_z^{2r} (\log x_1)^r}
\ll  
 \frac{(\log z)^r}{z } 
\frac{X^{2dr}} {W_z^{2r} (\log X)^r}
.$$
For all other $x_i$, we can assume 
that $x_1\leq  X^d/(\log X)^{2r}$ 
and bound trivially 
\begin{align*} &\sum_{\substack{ (\b m,\b m')  
\in \prod_{i=1}^{2r} (\N\cap [1,x_i]) 
\\ (\b m,\b m') \equiv (\b a,\b a') \md {W_z}  }} 
k_{\b s,\b s'}(\b m,\b m') -\rho(\b a,\b a',W_z)   
\int_{ \substack{   \prod_{i=1}^{2r} [1,x_i ]\\ 
x_1\leq  X^d/(\log X)^{2r}
}  }   \prod_{i=1}^{2r} \frac{1}{\sqrt{ \log t_i}} 
\mathrm d \b t  
\\
& \ll   W_z^{-2r}
\prod_{i=1}^{2r}  x_i 
+ (W_z\phi(W_z))^{-r}
\int_{ \substack{   \prod_{i=1}^{2r} [1,x_i ]\\ 
x_1\leq  X^d/(\log X)^{2r}
}  }  1 \mathrm d \b t . 
\end{align*}
This is acceptable since 
$\prod_{i=1}^{2r}  x_i \leq 
 X^{2dr} (\log X)^{-2r} \leq 
z^{-1} X^{2dr} (\log X)^{-r} $, where we used the fact that 
we took $z$ to be 
$[(\log X)/(2 \log \log X)]$.
\end{proof}

 \begin{lemma}\label
 {integral RV 580} There exists a strictly positive 
 constant $c=c(\b f,\b g )$ such that 
   $$
\frac{\c C_{\b s,\b s'}(X)}{X^{n}}=  I_X  S_z
+O\left(  \frac{ 1}{(\log X)^{r+c}}\right).$$
\end{lemma}\begin{proof}The bound 
$\| k_{\b s,\b s'}\|_1 \ll X^{2dr}(\log X)^{-r}$ 
can be proved by a straightforward application of the large sieve
as in the proof of 
the bound for $\Gamma_3$ in 
Lemma \ref{lem_hilbert_symbol_bound} if $W_z$ 
conditions were not present. Hence, by Lemma \ref{vivaldi RV 580}
the error term in \eqref{eq_dento8imi8iamiaxarapsixoula} 
is $$ \ll \frac{1}{(\log X)^r }  
\left( \frac{(\log \log X)^c}{(\log X)^c}
+\sum_{p\leq z} \frac{1}{p^{1+m(z)}} 
\right)  +
\frac{  (\log z)^r }{z (\log X)^{r}}
.$$ By \cite[Remark 2.6]{11388} we have 
$\sum_{p\leq z} p^{-1-m(z)} \leq 3 \cdot 2^{-m(z)}=
3 \cdot 2^{-[\log z] } \leq 6 \cdot z^{-\log 2}$, hence, 
the error term is acceptable upon 
replacing $c$ by $\min\{c,\log 2\}/2$.
\end{proof}For $d\in \mathbb N$  and $x\in \mathbb R\setminus (-1,1)$
the $d$-th Chebychev polynomial $T_d$ is defined as 
$$T_d(x):=
\frac{1}{2} ( (x+\sqrt{x^2-1} )^d + (x-\sqrt{x^2-1} )^d ).$$
This still makes sense when  $|x|<1 $  by 
writing $\sqrt{x^2-1}=\sqrt{-1} \sqrt{1-x^2}$
and letting $x=\cos \theta$ so that one ends up with 
$T_d(x)= \cos (d\theta) $.
The following lemma is the generalization 
of Remez's inequality to multivariate polynomials by 
Brudnyi--Ganzburg \cite[Theorem 2]{ganzburg}.
\begin{lemma}[Brudnyi--Ganzburg]\label{lem_ganzburg}
Let $K \subseteq \mathbb{R}^n$ be a convex body and $P\in 
\mathbb R[x_1,\ldots, x_n]$ of degree $d$. 
For any Lebesgue measurable subset $E \subseteq K$ 
with Lebesgue measure $\mu(E) > 0$,  we have 
$$
\sup\{|P(\b x )|: \b x \in K\} \le T_d\left( \frac{\mu(K)^{1/n} 
+ (\mu(K) - \mu(E))^{1/n}}
{\mu(K)^{1/n} - (\mu(K) -
\mu(E))^{1/n}} \right) 
\sup\{P(\b x )|: \b x \in E\}
.$$ 
\end{lemma} 
 \begin{corollary}\label{corolaramez}For 
 $d,n\in \mathbb N$,
 $\epsilon\in (0,1)$ and a non-zero polynomial
 $P\in \mathbb R[x_1,\ldots, x_n]$ the Lebesgue measure of $\b t \in [-1,1]^n$ satisfying 
 $|P(\b t )|\leq \epsilon$ is 
 at most $ 4nB^{-1/d}\epsilon^{1/d}$,
 where  $$B = \max \{ |P(\b t)| : \b t \in [-1,1]^n \}.$$
 \end{corollary}
\begin{proof} By Lemma~\ref{lem_ganzburg} with  $K = [-1,1]^n$
and $E = \{\b t \in K : |P(\b t)| \leq \epsilon\}$
we get  $$ B\leq T_d\l(\frac{1+(1-\xi)^{1/n}}
{1-(1-\xi)^{1/n}}\r) \epsilon,$$
where   
$\xi:=\mu(E)2^{-n}$. We assumed that  $\mu(E) > 0$, as otherwise 
the proof is trivial. By the definition of $T_d$ 
we see that 
$$T_d\l(\frac{1+(1-\xi)^{1/n}}
{1-(1-\xi)^{1/n}}\r) \leq  2^d
\l(\frac{1+(1-\xi)^{1/n}} {1-(1-\xi)^{1/n}}\r) ^d
\leq   \frac{4^d} { (1-(1-\xi)^{1/n})^d }.$$
By Bernoulli's inequality we have 
$(1-\xi/n)^n \geq 1-\xi$, hence, $1-(1-\xi)^{1/n}
 \geq \xi/n $. Thus,  
 \[B \leq  \frac{4^d} { (\xi/n)^d  }\epsilon.\qedhere\]
\end{proof}

 \begin{lemma}\label{ramez}We have 
 $$ I_X= 
\frac{\mathrm{vol}\l(\b t \in [-1,1]^n:
\min_i\{ s_i f_i(\b t ), s'_i g_i(\b t )\}>
0 \ \forall i\r)}{(\pi\log X) ^r}
 \left(1+O\left(\frac{1}{\log X}\right)\right)
,$$ where the implied constant depends 
only on $n,d,\b f$ and $\b g $.
\end{lemma}\begin{proof} We start by writing 
$$I_X=   \pi^{-r}
\int\limits_{  \b t \in [-1,1]^n, \eqref{ksexasatonumeromou} }
\prod_{i=1}^r 
\frac{1}{\sqrt{\log(X^d  s_i  f_i(\b t ) )}}
\frac{1}{\sqrt{\log(X^d  s'_i  g_i(\b t ) )}}
\mathrm d \b t ,$$ where
\begin{equation}\label{ksexasatonumeromou}
\begin{cases}
\min_i s_i f_i(\b t )>X^{-d},  \ \ \  
\min_i s'_i g_i(\b t )>X^{-d}
\\[1ex]
\max_i \{X s_i f_i(\b t ) ,X s'_i g_i(\b t )\} 
\log (3X\min_i  s_i f_i(\b t )) \leq  
(\min \{X s_i f_i(\b t ) ,X s'_i g_i(\b t )\})^{1+\frac{1}{16r}}.  
\end{cases}
\end{equation}   
For a constant $\gamma>0$ we have 
$$ \frac{1}{\sqrt{\log (X^d \gamma)}}=
 \frac{1}{\sqrt{d\log X }}
 \left(1+O_\gamma\left(\frac{1}{\log X}\right)\right),
$$hence   $$I_X=
\frac{\mathrm{vol}(\b t \in [-1,1]^n, 
\eqref{ksexasatonumeromou})}{(d\pi)^{r}(\log X)^r} 
\left(1+O\left(\frac{1}{\log X}\right)\right)
,$$ where the implied constant depends 
only on $\b f,\b g $.
By Corollary \ref{corolaramez}
we have 
\begin{equation}\label{eq_prolabainwk}
\mathrm{vol}\l(
\b t \in [-1,1]^n: \min
\{ |f_i(\b t )|,
|g_i(\b t )|\}\leq  (\log X)^{-d} \textrm{ for some } i\r)
\ll_{\b f,\b g} \frac{1}{\log X},\end
{equation}
hence, we may   
restrict to $\b t $ such that 
$ (\log X)^{-d} \leq 
 \min\{ |f_i(\b t )|,
|g_i(\b t )|\}$ holds for all $i$
by introducing a negligible error term.
On this set we have $$
\max_i \{X s_i f_i(\b t ) ,X s'_i g_i(\b t )\} 
\log (3X\min_i  s_i f_i(\b t )) \ll_{\b f,\b g }
X (\log X)
,$$ which for all large enough $X$ is at most 
$$ \frac{ X^{1+\frac{1}{16r}}  }{(\log X)^{d(1+\frac{1}{16r})}  }
\leq X^{1+\frac{1}{16r}}
(\min \{X s_i f_i(\b t ) ,X s'_i g_i(\b t )\})^{1+\frac{1}{16r}}.$$
Hence, $\mathrm{vol}(\b t \in [-1,1]^n, 
\eqref{ksexasatonumeromou})$ can be written as $$
\mathrm{vol}\l(\b t \in [-1,1]^n:
\min_i\{ s_i f_i(\b t ), s'_i g_i(\b t )\}>
(\log X)^{-d} \ \forall i\r)
\left(1+O\left(\frac{1}{\log X}\right)\right)$$
and employing \eqref{eq_prolabainwk} concludes the proof.
\end{proof}

Bringing together 
Lemma \ref{integral RV 580} 
and Lemma \ref{ramez} shows the following result:
 \begin{lemma}\label
 {integrafinal} There exists a strictly positive 
 constant $c=c(\b f,\b g )$ such that 
   $$
\frac{\c C_{\b s,\b s'}(X)}{X^{n}}=     S_z
\frac{\mathrm{vol}\l(\b t \in [-1,1]^n:
\min_i\{ s_i f_i(\b t ), s'_i g_i(\b t )\}>
0 \ \forall i\r)}{(d\pi\log X) ^r} 
+O\left( \frac{|S_z|}{(\log X)^{r+1}}+
\frac{1}{(\log X)^{r+c}}\right).$$
\end{lemma}

\section{Proof of main theorems} \label{s:finfin}
\subsection{Proof of Theorem \ref{thm:main2}}
\label{s-prf-thrm23}
By Lemma \ref{integrafinal} 
it remains to show that $S_z$
converges as $z\to \infty$. 
This would be straightforward if one could 
write $S_z$    as a product of $p$-adic sums 
for all primes $p\leq z$, however, 
this is prohibited 
due to the presence of the product $\prod_{i=1}^r$
in \eqref{vivaldi-griselda}. Instead, we open up the product as  
\begin{equation}\label{Traetta  Ifigenia in Tauride}
S_z= \sum_{\c A \subset \{1,\ldots, r\}}
\l(
 \prod_{i \in \c A} 
(s_i ,s'_i)_{\R}  \r) \prod_{p\leq z } \sigma_p(m(z),\c A),
\end{equation} where for a prime $p$ and a
non-negative integer $k$ we denote  
\begin{equation}\label{gatakisekamba}
\sigma_p(k,\c A):= 
\frac{(p-1)^r}{p^{r+kn}} 
 \sum_{\substack{ \b t \in (\Z/p^{k}\Z )^{n}, \ 
  \prod_{i=1}^r (f_i(\b t ), g_i(\b t ))_{\mathbb{Q}_p} = 1
  \\
 \max\{v_p(f_i(\b t)), v_p(g_i(\b t))\} \leq 
k - 1 - 2\mathds{1}_{\{2\}}(p),  \forall i
} }  \prod_{i \in \c A}  
(f_i(\b t),g_i(\b t))_{\Q_p}.\end{equation} 
Each set $\c A$ will give rise to a different
Euler product of $p$-adic densities.

For a prime $p$ let  $\mu_p$ be the $p$-adic Haar measure normalised 
so that $\mu_p(\Z_p)=1$. We define  $$\sigma_p(\c A):=
\frac{p^r}{(p-1)^{r}} 
\int_{ \substack{\b t \in \Z_p^{2n} \\
\prod_{i=1}^r (f_i(\b t ), g_i(\b t ))_{\mathbb{Q}_p} = 1}} \ \  
\prod_{i \in \c A}  
(f_i(\b t),g_i(\b t))_{\Q_p}
\mathrm d\mu_p $$ and we show that the integral 
converges.
\begin{lemma}\label{mlimit}
The integral $\sigma_p(\c A)$ is well-defined.
For each prime $p$ and integer $k\geq 2 $
we have 
$$\sigma_p(k,\c A)= \sigma_p(\c A) +O(p^{-k}),$$
where      the implied constant depends 
only on $\b f$ and $\b g $.\end{lemma}
\begin{proof}
We reformulate \eqref{gatakisekamba} 
by freezing the values of $ f_i (\b t ),
g_i (\b t )$ to get 
$$\sigma_p(k,\c A)
\frac{(p-1)^{r}}{p^{r}} =   p^{-2kr}
 \sum_{\substack{ (\b a,\b a') \in (\Z/p^{k}\Z )^{2r}, \ 
  \prod_{i=1}^r (a_i, a'_i)_{\mathbb{Q}_p} = 1
  \\
 \max\{v_p(a_i), v_p(a'_i)\} \leq 
k- 1 - 2\mathds{1}_{\{2\}}(p),  \forall i
} }  
N_{p^k}(\b a,\b a')
\prod_{i \in \c A}  
(a_i,a'_i)_{\Q_p},$$ where 
$$N_{p^k}(\b a,\b a'):=p^{-k(n-2r)} 
\#\l\{\b t \in (\Z/p^{k}\Z )^{n}: (\b f(\b t ), 
\b g (\b t )) \equiv (\b a , \b a')\md{p^k} \r\}.$$
Assume that $p\neq 2$
and write $a_i=p^{\alpha_i} u_i$
$a'_i=p^{\alpha'_i} u'_i$ so that the right-hand side
  becomes \begin{align*}p^{-2kr}
\sum_{\boldsymbol \alpha, 
\boldsymbol \alpha'\in [0,k-1]^r }
\sum_{\b u,\b u'} 
\mathds 1_{\{1\}}   &\l(  \prod_{i=1}^r 
\l(\frac{-1}{p}\r)^{\alpha_i \alpha'_i}
\l(\frac{u'_i}{p}\r)^{\alpha_i}
\l(\frac{u_i}{p}\r)^{\alpha'_i} \r) \\
\times &
N_{p^k}((p^{\alpha_i} u_i,p^{\alpha'_i} u'_i))
\prod_{i \in \c A}  
\l(\frac{-1}{p}\r)^{\alpha_i \alpha'_i}
\l(\frac{u'_i}{p}\r)^{\alpha_i}
\l(\frac{u_i}{p}\r)^{\alpha'_i}
,\end{align*} where   $\b u , \b u'$ 
run over 
$u_i \in (\Z/p^{k-\alpha_i}\Z)^*$
and $u'_i \in (\Z/p^{k-\alpha'_i}\Z)^*$
for all $i =1,\ldots, r$.  

The quadratic symbols 
 depend only on the parity of $\alpha_i,\alpha'_i$
and the reduction of $u_i,u'_i $ in $\F_p$.
Hence the last expression can be written as 
\begin{align*}p^{-2kr}
&
\sum_{\boldsymbol \beta , \boldsymbol \beta'\in \{0,1\}^r }
\sum_{ \substack{  \b t , \b t' \in (\F_p^*)^r    }} 
\mathds 1_{\{1\}} \l(  \prod_{i=1}^r 
\l(\frac{-1}{p}\r)^{\beta_i \beta'_i}
\l(\frac{t'_i}{p}\r)^{\beta_i}
\l(\frac{t_i}{p}\r)^{\beta'_i} \r)
\prod_{i \in \c A}  
\l(\frac{-1}{p}\r)^{\beta_i \beta'_i}
\l(\frac{t'_i}{p}\r)^{\beta_i}
\l(\frac{t_i}{p}\r)^{\beta'_i} 
\\  \times  &
\sum_{ \substack{ 
\boldsymbol \alpha , \boldsymbol \alpha'\in [0,k-1]^r\\
(\boldsymbol \alpha , \boldsymbol \alpha')\equiv 
(\boldsymbol \beta , \boldsymbol \beta') \md 2 }}
 \sum_{ \substack{  
  u_i \in \Z/p^{k-\alpha_i}\Z, \ 
u_i\equiv t_i \md p, \forall i \\ 
 u'_i \in (\Z/p^{k-\alpha'_i}\Z, \
 u'_i\equiv t'_i \md p, \forall i   }}  
 N_{p^k}((p^{\alpha_i} u_i,p^{\alpha'_i} u'_i)) 
.\end{align*} We can now extend the sum to an infinite series 
by adding all terms with  
$\boldsymbol \alpha , \boldsymbol \alpha'\in [0,\infty)^r$.
By  \eqref{eq:singseries} we have 
  $N_{p^k}(\b a,\b a') =O(1)$
with implied constants only depending on $\b f,\b g$,  hence, 
the tail is 
$$\ll_{r,\b f,\b g} p^{-2kr} p^{2r}
\sum_{ \substack{ 
\boldsymbol \alpha , \boldsymbol \alpha'\in [0,\infty)^r\\
\max_i\{\alpha_i,\alpha'_i\} \geq k }}
p^{2rk-2r-\sum_i(\alpha_i+\alpha'_i)}
 \ll p^{-k} .$$ Hence, 
 $\lim_{k\to \infty} \sigma_p(k,\c A)$ exists and  
 $$\sigma_p(k,\c A)= \lim_{k\to\infty}
 \sigma_p(\c A) +O(p^{-k})$$ 
 with an implied constant that is independent of $p$ and $k$. We clearly have
$$\sigma_p(k,\c A)= 
\frac{p^r}{(p-1)^{r}} 
\int_{ \substack{\b t \in \Z_p^{2n} \\
\prod_{i=1}^r (f_i(\b t ), g_i(\b t ))_{\mathbb{Q}_p} = 1}} \ \  
\prod_{i \in \c A}  
(f_i(\b t),g_i(\b t))_{\Q_p}
\mathrm d\mu_p $$
where the integral is subject to $
\max_i\{v_p(f_i(\b t)), v_p(g_i(\b t))\} 
\leq  k-1$. Letting $k\to\infty$ concludes the proof 
for $p\neq 2$. The prime $p=2$ is treated similarly
by using the analogous explicit formulas for the 
Hilbert symbol in $\mathbb Q_2$.
\end{proof}
\begin{lemma}\label{leonardoleonataleoratorio}
There exists $c=c(\b f,\b g )>1$ such that 
for each prime $p$  we have 
$$\sigma_p(\c A)= 1 +O\l(p^{-c}\r),$$
where      the implied constant depends 
only on $\b f$ and $\b g $. 
\end{lemma}
\begin{proof}
By Lemma \ref{mlimit} we have 
$\sigma_p(\c A)=\sigma_p(2,\c A)+O(p^{-2})$.
We may plainly assume that $p\neq 2 $
so that by the same argument as in the proof of 
  Lemma \ref{mlimit} we have 
\begin{align*} \sigma_p(2,\c A)=
&\frac{1}{(p-1)^r p^{3r}}  
\sum_{\boldsymbol \alpha, 
\boldsymbol \alpha'\in \{0,1\}^r }
\sum_{\b u,\b u'} 
\mathds 1_{\{1\}}   \l(  \prod_{i=1}^r 
\l(\frac{-1}{p}\r)^{\alpha_i \alpha'_i}
\l(\frac{u'_i}{p}\r)^{\alpha_i}
\l(\frac{u_i}{p}\r)^{\alpha'_i} \r) \\
\times &
N_{p^2}((p^{\alpha_i} u_i,p^{\alpha'_i} u'_i))
\prod_{i \in \c A}  
\l(\frac{-1}{p}\r)^{\alpha_i \alpha'_i}
\l(\frac{u'_i}{p}\r)^{\alpha_i}
\l(\frac{u_i}{p}\r)^{\alpha'_i}
,\end{align*} where   $\b u , \b u'$ 
run over 
$u_i \in (\Z/p^{2-\alpha_i}\Z)^*$
and $u'_i \in (\Z/p^{2-\alpha'_i}\Z)^*$
for all $i =1,\ldots, r$.
By  \eqref{eq:singseries} we have 
$N_{p^2}(\b a,\b a') =1+O(p^{-c})$ 
for some $c=c(\b f,\b g)>1$ uniformly in $\b a,\b a'$.
We   get   \begin{align*} \sigma_p(2,\c A)=O(p^{-c})+
\frac{1}{(p-1)^r p^{3r}}  &
\sum_{\boldsymbol \alpha, 
\boldsymbol \alpha'\in \{0,1\}^r }
\sum_{\b u,\b u'} 
\mathds 1_{\{1\}}   \l(  \prod_{i=1}^r 
\l(\frac{-1}{p}\r)^{\alpha_i \alpha'_i}
\l(\frac{u'_i}{p}\r)^{\alpha_i}
\l(\frac{u_i}{p}\r)^{\alpha'_i} \r)  \\
& \times 
\prod_{i \in \c A}  
\l(\frac{-1}{p}\r)^{\alpha_i \alpha'_i}
\l(\frac{u'_i}{p}\r)^{\alpha_i}
\l(\frac{u_i}{p}\r)^{\alpha'_i}
.\end{align*}
The cases with $\sum (\alpha_i+\alpha'_i) \geq  2$ 
trivially  contribute  to 
the sum above
a quantity that is in modulus 
$$ \ll  \frac{1}{
(p-1)^r p^{3r} }
\sum_{\boldsymbol \alpha, 
\boldsymbol \alpha'\in \{0,1\}^r }
  p^{4r-\sum (\alpha_i+\alpha'_i)}
\ll \frac{1}{p^2}.$$ The cases 
$\sum (\alpha_i+\alpha'_i) =0 $ 
  contribute $$ \frac{1}{(p-1)^r p^{3r} } 
\phi(p^2)^{2r}=1-\frac{r}{p }+O\l(\frac{1}{p^2}\r).$$
In the remaining cases,
$\sum (\alpha_i + \alpha'_i) = 1$, hence, 
all variables $\alpha_i$ and $\alpha'_i$ 
are zero except for one.  
When the index $j$ satisfying
$\alpha_j+\alpha'_j = 1$
is in $\c A$ the contribution is  
$$
\frac{1}{(p-1)^r p^{3r}}   
\sum_{j\in \c A} 
 \phi(p^2)  ^{2r-1}\Bigg(
\sum_{\substack{ u_j \in \F_p^* \\  
(\frac{u_j}{p})=1 }} 1 
+\sum_{\substack{ u'_j \in \F_p^* \\  
(\frac{u'_j}{p})=1 }} 1 \Bigg)  =
\frac{\#\c A}{p} \l(1-\frac{1}{p}\r)^r=
\frac{\#\c A}{p} +O\l(\frac{1}{p^2}\r).$$
If $j \notin \c A$ then we get the contribution 
$$
\frac{1}{(p-1)^r p^{3r}}  
\sum_{j\notin \c A}  \phi(p^2)  ^{2r-1}\Bigg(
\sum_{\substack{ u_j \in \F_p^* \\  
(\frac{u_j}{p})=1 }} 1 
+\sum_{\substack{ u'_j \in \F_p^* \\  
(\frac{u'_j}{p})=1 }} 1 \Bigg)  =
\frac{(r-\#\c A)}{p} \l(1-\frac{1}{p}\r)^r=
\frac{(r-\#\c A)}{p} +O\l(\frac{1}{p^2}\r)
.$$ Adding up the contributions we obtain 
the overall estimate 
\[\sigma_p(2,\c A)=O(p^{-{\min\{c,2\}}})+
\l(1-\frac{r}{p }\r)
+\l(\frac{\#\c A}{p }\r)
+\l(\frac{(r-\#\c A)}{p }\r)
=1+O(p^{-{\min\{c,2\}}}).\qedhere\]
\end{proof}
\begin{lemma}\label{Atto Primo Scena 6 No 5 Aria Paisiello: Nina O Sia La Pazza Per Amore}
There exists $\gamma =\gamma (\b f,\b g )>0$ such that 
for all $z\geq 1 $ we have 
$$S_z= \sum_{\c A \subset \{1,\ldots, r\}} \l(
 \prod_{i \in \c A} (s_i ,s'_i)_{\R}  \r)
\prod_{\substack{ p  \textrm{ prime} \\ p=2}}^\infty 
\sigma_p(\c A)
+O(z^{-\gamma}),$$
where      the implied constant depends 
only on $\b f$ and $\b g $. 
\end{lemma}
\begin{proof}By the inequality $$\l| \prod_{p\leq z}
\sigma_p(m(z),\c A)-\prod_{p\leq z}
\sigma_p(\c A)\r| \leq 
\sum_{\substack{ q  \textrm{ prime} \\ q\leq z }}
|\sigma_p(m(z),\c A)-\sigma_p(\c A)|
\prod_{\substack{ p\leq z \\ p\neq q }}
\max\{\sigma_p(m(z),\c A),\sigma_p(\c A)\}
$$ and Lemmas \ref{mlimit}-\ref{leonardoleonataleoratorio}
we obtain 
$$\l| \prod_{p\leq z}
\sigma_p(m(z),\c A)-\prod_{p\leq z}
\sigma_p(\c A)\r| \ll \sum_{q\leq z } q^{-m(z)}.$$
By \cite[Remark 2.6]{11388} this is $\ll 2^{-m(z)}$. 
Since we have chosen   $m(z)=[\log z]$ this gives
$$\prod_{p\leq z}
\sigma_p(m(z),\c A)=\prod_{p\leq z}
\sigma_p(\c A)+O(z^{-\log 2}).$$
We can extend the product to all primes $p$ in the right-hand side 
by employing Lemma \ref{leonardoleonataleoratorio}. 
Together with \eqref{Traetta  Ifigenia in Tauride} this 
shows that $$ S_z=\Bigg( \sum_{\c A \subset \{1,\ldots, r\}}
\l( \prod_{i \in \c A}  (s_i ,s'_i)_{\R}  \r)
\prod_{p=2}^\infty
\sigma_p(m(z),\c A) \Bigg)+O(z^{-{\gamma_0}})$$ for some strictly 
positive constant $\gamma_0$ that depends only on $\b f$ and $\b g $.\end{proof}
Recalling that we have chosen $z=[(\log X)/(\log \log X)]$ proves that  
$$S_z= \sum_{\c A \subset \{1,\ldots, r\}} \l(
 \prod_{i \in \c A} (s_i ,s'_i)_{\R}  \r)
\prod_{\substack{ p  \textrm{ prime} \\ p=2}}^\infty 
\sigma_p(\c A)
+O\l((\log X)^{-\gamma}\r),$$ hence, 
Lemma \ref{integrafinal} gives 
   $$\frac{\c C_{\b s,\b s'}(X)}{X^{n}}=
   \frac{\c M (\b s,\b s') }{(d\pi\log X) ^r}
+O\left( \frac{1}{(\log X)^{r+c}}\right),$$
 where $$ \c M(\b s,\b s')    :=
 \mathrm{vol}\l(\b t \in [-1,1]^n:
\min_i\{ s_i f_i(\b t ), s'_i g_i(\b t )\}>
0 \ \forall i\r)
\sum_{\c A \subset \{1,\ldots, r\}} \l(
 \prod_{i \in \c A} (s_i ,s'_i)_{\R}  \r)
\prod_{\substack{ p  \textrm{ prime} \\ p=2}}^\infty 
\sigma_p(\c A).$$ Adding this back to \eqref{eq_ser} 
shows that 
$$N_\alpha(X)=   \frac{\c M}{(d\pi\log X) ^r} 
+O\left( \frac{1}{(\log X)^{r+c}}\right),$$
 where $$\c M= \sum_{\c A \subset \{1,\ldots, r\}} 
 \sigma_\infty(\c A)  
  \prod_{\substack{ p  \textrm{ prime} \\ p=2}}^\infty 
\sigma_p(\c A) ,$$ with 
$$ \sigma_\infty(\c A)  =
\sum_{\substack{ \b s,\b s'\in \{-1,1\}^r\\ 
\prod_{i=1}^r (s_i,s'_i)_\R=1 }}
 \mathrm{vol}\l(\b t \in [-1,1]^n:
\min_i\{ s_i f_i(\b t ), s'_i g_i(\b t )\}>
0 \ \forall i\r) \prod_{i \in \c A} (s_i ,s'_i)_{\R}  .$$
This concludes the proof of the asymptotic in
Theorem \ref{thm:main2}. It remains to prove 
the various expressions for the leading constant 
that are claimed in Remark \ref{rem_constant2}.

  So far we have seen that 
the leading constant is given by 
$$
 \frac{1}{(d\pi)^{\Delta(\alpha)}}
  \sum_{\c A \subset \{1,\ldots, r\}} 
 \sigma_\infty(\c A)  
  \prod_{\substack{ p  \textrm{ prime} \\ p=2}}^\infty 
  \frac{\sigma_p(\c A) }{( 1-\frac{1}{p} )^{\Delta(\alpha)}}
,$$ with 
$$\sigma_v(\c A) =\int_{\substack{\mathbf{t} \in \mathcal{O}_v^n \\ \prod_{i=1}^r (f_i(\mathbf{t}), g_i(\mathbf{t}))_v = 1}} \left( \prod_{i \in \mathcal{A}} (f_i(\mathbf{t}), g_i(\mathbf{t}))_v \right) \mathrm{d}\mu_v(\mathbf{t})
$$ for $v\in \{\infty\}\cup\{p\textrm{ prime}\}$,
where $\c O_\infty=[-1,1]$, $\c O_p=\Z_p$,
 $(\cdot, \cdot)_v$ is the local Hilbert symbol on 
 $\mathbb{Q}_v$
 and $\mu_v$ is the normalised 
 Haar measure on $\mathbb{Q}_v^n$ 
 satisfying $\mu_v(\mathcal{O}_v^n) = 1$.

Let $\Omega_T=[-1,1]^n \times \prod_{p\leq T}\Z_p^n$,
let $\mu_\infty$ be the standard Lebesgues measure on $[-1,1]$,
let $\mu_p$ be the standard $p$-adic Haar measure on $\Z_p^n$
and let $\nu_T:=\mu_\infty \cdot \prod_{p\leq T}\mu_p$ be the
correspondent measure on $\Omega_T$.
Denote $S_T :=\{\infty\} \cup \{p\leq T \}$.
Then we can write the leading constant as 
$$ 
 \frac{1}{(d\pi)^{\Delta(\alpha)}}
 \lim_{T\to \infty} \prod_{p\leq T} 
( 1-\frac{1}{p} )^{-\Delta(\alpha)}
\int_{\substack{\mathbf{t} \in \Omega_T \\ 
\prod_{i=1}^r (f_i(\mathbf{t}), g_i(\mathbf{t}))_{v} = 1
\forall v \in S_T}}  
  \sum_{\c A \subset \{1,\ldots, r\}} 
  \prod_{i \in \mathcal{A}} 
\left( \prod_{v\in S_T}(f_i(\mathbf{t}), g_i(\mathbf{t}))_v \right) \mathrm{d}\mu_v(\mathbf{t})
,$$ and since $  \sum_{\c A \subset \{1,\ldots, r\}} 
  \prod_{i \in \mathcal{A}} \lambda_i
  =\prod_{i=1}^r (1+ \lambda_i)$ 
  we can write the function inside the limit 
  as $$ 
 \prod_{p\leq T} 
( 1-\frac{1}{p} )^{-\Delta(\alpha)}
\int_{\substack{\mathbf{t} \in \Omega_T \\ 
\prod_{i=1}^r (f_i(\mathbf{t}), g_i(\mathbf{t}))_{v} = 1
\forall v \in S_T}}  
  \prod_{i =1}^r  \left(1+
\left( \prod_{v\in S_T}(f_i(\mathbf{t}), g_i(\mathbf{t}))_v \right)\right) \mathrm{d}\mu_v(\mathbf{t})
.$$ The new product vanishes except when for each $i=1,\ldots, r$ one has $$\prod_{v\in S_T}(f_i(\mathbf{t}), g_i(\mathbf{t}))_v =1 .$$ Thus, the integral equals 
$$ 2^r \nu_T\left( \b t \in \Omega_T: 
\prod_{i=1}^r (f_i(\mathbf{t}), g_i(\mathbf{t}))_{v} = 1
\forall v \in S_T,  \prod_{v\in S_T}(f_i(\mathbf{t}), g_i(\mathbf{t}))_v  =1 \forall i=1,\ldots, r
\right)
.$$ This can be written as 
$$ 2^r \sum_{\substack{ \b s , \b s' \in \{-1,1\}^r \\ 
\prod_{i=1}^r (s_i,s'_i)_\R=1 }}
\gamma(\b s , \b s ')
\mathrm{vol}(\b t \in [-1,1]^n : 
\mathrm{sign}(f_j(\b t)= s_j, 
\mathrm{sign}(g_j(\b t)= s'_j \forall j \in \{1,\ldots, r\}
)),
$$
where $\gamma(\b s , \b s ') $ is 
$$
 \lim_{T \to \infty}  
\l[
\prod_{\substack{p \ \mathrm{ prime} \\ p \leq T} }
\left(1-\frac1p\right)^{-\Delta(\alpha)} 
\hspace{-0.2cm} 
\mu_p
\r]
\hspace{-0.1cm}
\l \{\b t \in   
\prod_{p \leq T} \Z_p^{n+1}
\colon \ 
\hspace{-0.2cm}
\begin{array}{l} 
\prod_{i=1}^r (f_i(\b t ),  g_i(\b t ) )_{\Q_p}=1 \ 
\forall p \leq T,  \\ 
\prod_{p\leq T} (f_i(\b t ),  g_i(\b t ) )_{\Q_p}
=(s_i,s'_i)_\R  \  \forall i\in \{1,\ldots,R\}\end{array}
\hspace{-0.1cm}
\right\}
.$$

\subsection{Proof of Theorem \ref{thm_main3}}  
\label{s-prf-thrm3}
To estimate $$
  \#\left\{\b n, \b n' \in \mathbb Z^r: 
0<|n_i|,|n'_i|\leq X \ \forall i, \ \  
1=\prod_{i=1}^r (n_i,n'_i)_{\R}=
\prod_{i=1}^r (n_i,n'_i)_{\Q_p} \ 
\forall  \ \mathrm{prime }\  p
\r\}$$
 we start by splitting it according to the signs 
 as $$\sum_{\substack{ \b s,\b s'\in \{-1,1\}^r\\ 
\prod_{i=1}^r (s_i,s'_i)_\R=1 }}  
\sum_{\b n,\b n' \in \N^r \cap X[0,1]^r }
k_{\b s, \b s'}((s_1 n_1, \ldots, s_r n_r ),
(s'_1 n'_1, \ldots, s'_r n'_r ) ).
$$ Keeping the same 
definition of $W_z$ as in
\eqref{Porpora Torbido intorno al core} 
and taking $z=[(\log X)/(2 \log \log X)]$, 
$m(z)=[\log z]$ as in \S \ref{circle_method},
we split the inner sum in progressions 
modulo $W_z$  and apply
Proposition \ref{prop_main} with $x_i=X$ for all $i$.
Up to a negligible error term this yields 
$$\frac{X^{2r}}{(\pi  \log X)^r}
\sum_{\substack{ \b s,\b s'\in \{-1,1\}^r \\ 
\prod_{i=1}^r (s_i,s'_i)_\R=1 }}  \ 
\sum_{\substack{ \b a ,\b a' 
\in (\Z/W_z\Z)^r \\\eqref{eq_romanakos}}} 
\frac{\prod_{i=1}^r (1+\epsilon_i ) }
{( W_z\phi(W_z))^r} $$with
$\epsilon_i=(s_i ,s'_i)_{\R} 
\prod_{p\leq z} (s_ia_i,s'_ia'_i)_{\Q_p}$ and 
the sum over $\b a,\b a'$
is subject to  \begin{equation}
\label{eq_romanakos}
\max\{v_p(a_i),v_p(a'_i)\}\leq  m(z)-1-2
\mathds 1_{\{2\}}(p),
\ \ \ 
\prod_{i=1}^r (s_i a_i,s'_i a'_i )_{\Q_p}=1 \ \ 
\textrm{ for all primes} \  p\leq z. \end{equation}
The sum over $\b a,\b a'$
cannot be written as a product of $p$-adic densities 
since $\prod_{i} (1+\epsilon_i)$ is not multiplicative
in $p$. Instead, we write it as 
$$
\sum_{\c A \subset \{1,\ldots, r\}} 
\sum_{\substack{ \b a ,\b a' \in (\Z/W_z\Z)^r \\\eqref{eq_romanakos}}} 
(W_z \phi(W_z))^{-r}
\prod_{i \in \c A}
\epsilon_i=\sum_{\c A \subset \{1,\ldots, r\}} 
\l(\prod_{i \in \c A} (s_i,s'_i)_\R\r)
\prod_{p\leq z } n_p(\c A),$$ where 
$$ n_p(\c A):=(p^{m(z)}\phi(p^{m(z)}))^{-r}
\sum_{\substack{ \b a ,\b a' \in (\Z/p^{m(z)}\Z)^r,
\prod_{i=1}^r (s_i a_i,s'_i a'_i )_{\Q_p}=1
\\ 
\max\{v_p(a_i),v_p(a'_i)\}\leq  m(z)-1-2
\mathds 1_{\{2\}}(p) 
}} 
\prod_{i \in \c A}
(s_ia_i,s'_ia'_i)_{\Q_p}.$$
Arguments similar to those in the end of 
\S \ref{s-prf-thrm23} shows that this has a limit as $z\to \infty$, and,
furthermore, the leading constant has the Haar measure expression
\eqref{Leonardo Leo_Miserere concertato a due chori}.
The rest of the paper is thus devoted to proving the 
explicit formula and the lower bound 
for $c_r$.

\begin{lemma}\label{lem_anthpic}
Let $p\neq 2 $ be a prime and 
$\c A\subset \{1,\ldots, r\}$.
Then for $\alpha, \alpha'\in \{0,1\}$ 
we have $$\sum_{u,u'\in \F_p^*} 
(p^\alpha u, p^{\alpha'} u')_{\Q_p}=
(p-1)^{2}\begin{cases}
1,  & \alpha+\alpha'=0, \\
0,  & \alpha+\alpha'\neq 0.
\end{cases} $$\end{lemma}
\begin{proof}The proof is immediate by using  
$(p^\alpha u, p^{\alpha'} u')_{\Q_p}=
(\frac{-1}{p})^{\alpha \alpha'}
(\frac{u}{p})^{ \alpha'}
(\frac{u'}{p})^{\alpha}$.\end{proof}

\begin{lemma}\label{lem_anthpic2}
Let $p\neq 2 $ be a prime and 
$\c A\subset \{1,\ldots, r\}$. Then 
$$ n_p(\c A)=
 \l(1-\frac{1}{p}\r)^{-r} 
\frac{1}{2}
\l( \l(  1+\frac{1}{p} \r)^{-2 \#\c A}+
\l(  1+\frac{1}{p} \r)^{-2 (r-\#\c A)} \r) 
(1+O(p^{-m(z)})),
$$ where the implied constant depends at most on $r$.
\end{lemma}
\begin{proof} Using the identity 
$\mathds 1 (h=1) =\frac{1}{2} (1+h)$ for any 
$h\in \{1,-1\}$ we infer that  
\begin{equation}\label{arianistes} n_p(\c A)=
p^{-2rm(z)} \l(1-\frac{1}{p}\r)^{-r}
\l(\frac{K_p(\c A)+K_p(\c A^c)}{2}\r) 
,\end{equation}
where for a subset $\c S \subset \{1,\ldots, r\}$ we denote 
$$K_p(\c S):=
\sum_{\substack{ \b a ,\b a' \in (\Z/p^{m(z)}\Z)^r 
\\ \max\{v_p(a_i),v_p(a'_i)\}\leq  m(z)-1}} 
\prod_{i \in \c S}(s_ia_i,s'_ia'_i)_{\Q_p}.$$
We write  $a_i=p^{\alpha_i} u_i,a'_i=p^{\alpha'_i} u'_i$
with $p\nmid u_i,u'_i$  so that 
$$ K_p(\c S)= 
\sum_{\boldsymbol \alpha, \boldsymbol \alpha'\in [0,m(z)-1]^r }
\sum_{\substack{  
u_i \in (\Z/p^{m(z)-\alpha_i}\Z)^* \forall i\\ 
u'_i \in (\Z/p^{m(z)-\alpha'_i}\Z)^* \forall i }} 
\prod_{i \in \c S}(s_i p^{\alpha_i} u_i ,
s'_i p^{\alpha'_i} u'_i )_{\Q_p}.$$ 
Since the Hilbert symbols in the sums depend only on 
$u_i,u'_i \md p$ we see that $$ K_p(\c S)= p^{2r(m(z) -1) }
\sum_{\boldsymbol \alpha, \boldsymbol \alpha'\in [0,m(z)-1]^r }
p^{  -\sum_{i=1}^r (\alpha_i+\alpha'_i)}
\sum_{ \b u, \b u'\in (\F_p^*)^r } 
\prod_{i \in \c S}(s_i p^{\alpha_i} u_i ,
s'_i p^{\alpha'_i} u'_i )_{\Q_p}.$$ 
The sum over $\b u, \b u'$ splits as a product over all $1\leq i \leq r $ of sums over $u_i,u'_i$. When $i\notin S$ then 
there is no Hilbert symbol, hence, the sum equals $(p-1)^2$.
When $i\in \c S$ the sum vanishes by Lemma \ref{lem_anthpic} 
except when $2\mid (\alpha_i,\alpha'_i)$. Hence, 
$$ K_p(\c S)= p^{2r(m(z) -1) }(p-1)^{2r}
\sum_{\substack{
\boldsymbol \alpha, \boldsymbol \alpha'\in [0,m(z)-1]^r\\
i\in \c S\Rightarrow 2\mid (\alpha,\alpha'_i)
} }
p^{  -\sum_{i=1}^r (\alpha_i+\alpha'_i)}.$$ The sum over 
$\boldsymbol \alpha, \boldsymbol \alpha'$ equals 
$$
\l(  1-\frac{1}{p}  \r)^{-2 (r-\#\c S)}
\l(  1-\frac{1}{p^2} \r)^{-2 \#\c S}
(1+O(p^{-m(z)}))=\l(  1-\frac{1}{p} \r)^{-2 r}
\l(  1+\frac{1}{p} \r)^{-2 \#\c S}
(1+O(p^{-m(z)}))
,$$ hence, $$ K_p(\c S)= p^{2r m(z)  } 
\l(  1+\frac{1}{p} \r)^{-2 \#\c S}
(1+O(p^{-m(z)}))
.$$ Thus, by \eqref{arianistes} we conclude the proof.
\end{proof}

\begin{lemma}\label{lem_anthpicts2p}
For $\c A\subset \{1,\ldots, r\}$, $s,s'\in \{-1,1\}$
and  $\alpha, \alpha'\in \{0,1\}$ 
we have $$\sum_{u,u'\in (\Z/8\Z)^*} 
(2^\alpha su, 2^{\alpha'} s'u')_{\Q_2}=
\begin{cases} 8,  & \alpha+\alpha'=0, \\
0,  & \alpha+\alpha'\neq 0.
\end{cases} $$\end{lemma}
\begin{proof} The change of variables 
$(su,s'u')\equiv (x,x') \md 8$ is invertible, hence, the sum equals 
$$\sum_{x,x'\in (\Z/8\Z)^*} 
(2^\alpha x, 2^{\alpha'} x')_{\Q_2}.$$
If $\alpha =\alpha'=0$ then 
$(2^\alpha x, 2^{\alpha'} x')_{\Q_2}$ equals $1$ 
except when both $x,x'$ are $3\md 4$, hence, the value of 
the sum is $8$. If $\alpha=0, \alpha'=1$, then 
the sum becomes $$\sum_{x,x'\in (\Z/8\Z)^*} 
(-1)^{\frac{(x-1)(x'-1)}{4}}
\l(\frac{2}{x}\r)
= \sum_{x\in \{1,7\}, \x'\in (\Z/8\Z)^*} 
(-1)^{\frac{(x-1)(x'-1)}{4}}
-\sum_{x\in \{3,5\}, \x'\in (\Z/8\Z)^*} 
(-1)^{\frac{(x-1)(x'-1)}{4}}
$$ which equals $4-4=0$. In the last case 
$\alpha= \alpha'=1$, the sum can be written as $$ 
\sum_{x, x'\in (\Z/8\Z)^*} 
(-1)^{\frac{(x-1)(x'-1)}{4}}
\l(\frac{2}{x}\r) \l(\frac{2}{x'}\r)
.$$ For each fixed $x'\in (\Z/8\Z)^*$
the function $f(x)=(-1)^{\frac{(x-1)(x'-1)}{4}}
(\frac{2}{x}) $ is a non-principal character 
modulo $8$, hence, the sum above vanishes.\end{proof}

\begin{lemma}\label{lem_anthpic234}
For all $\c A\subset \{1,\ldots, r\}$ we have 
$$ n_2(\c A) =  2^{r-1}
 ((2/9)^{-\#\c A}+(2/9)^{r-\#\c A} ) 
 (1+O(2^{-m(z)})),
$$ where the implied constant depends at most on $r$.
\end{lemma}
 \begin{proof}
 We start by writing  
\begin{equation}\label{divine aria_Atto Primo: Scena 6 - No. 5 Aria (Paisiello)  Nina O Sia La Pazza Per Amore} n_2(\c A)=
2^{-2rm(z)+r}\l(\frac{N_2(\c A)+N_2(\c A^c)}{2}\r) 
,\end{equation}
where for a subset $\c S \subset \{1,\ldots, r\}$ we denote 
$$N_2(\c S):=
\sum_{\substack{ \b a ,\b a' \in (\Z/2^{m(z)}\Z)^r 
\\ \max\{v_2(a_i),v_2(a'_i)\}\leq  m(z)-3}} 
\prod_{i \in \c S}(s_ia_i,s'_ia'_i)_{\Q_2}.$$
We can then express this as  
$$ N_2(\c S)= 
\sum_{\boldsymbol \alpha, \boldsymbol \alpha'\in [0,m(z)-3]^r }
\sum_{\substack{  
u_i \in (\Z/2^{m(z)-\alpha_i}\Z)^* \forall i\\ 
u'_i \in (\Z/2^{m(z)-\alpha'_i}\Z)^* \forall i 
}} 
\prod_{i \in \c S}
(s_i 2^{\alpha_i} u_i ,s'_i
2^{\alpha'_i} u'_i )_{\Q_2}.$$ 
By standard properties of the Hilbert symbol
in $\Q_2$, they only depend on $u_i,u'_i \md 8$. 
Hence, $$ N_2(\c S)= 
\sum_{\boldsymbol \alpha, \boldsymbol \alpha'\in [0,m(z)-3]^r }
2^{2r(m-3)-\sum_{i=1}^r (\alpha_i+\alpha'_i)}
\sum_{\substack{  
u_i \in (\Z/8\Z)^* \forall i\\ 
u'_i \in (\Z/8\Z)^* \forall i 
}} \prod_{i \in \c S}
(s_i 2^{\alpha_i} u_i ,s'_i
2^{\alpha'_i} u'_i )_{\Q_2}.$$ 
The inner sum over $\b u, \b u'$
splits as a product of sums over pairs $(u_i,u'_i)$.
By Lemma \ref{lem_anthpicts2p} when $i\in \c S$ 
this sum vanishes except when both $\alpha_i,\alpha'_i$ 
are even. Hence,  $$ N_2(\c S)= 2^{2rm(z)-6r}
\sum_{\substack{
\boldsymbol \alpha, \boldsymbol \alpha'\in [0,m(z)-3]^r \\
i\in \c S \Rightarrow 2\mid (\alpha_i,\alpha'_i)}}
2^{ -\sum_{i=1}^r (\alpha_i+\alpha'_i)} 8^{\#\c S}
16^{r-\#\c S}= 
2^{2rm(z)}  (2/9)^{\#\c S} (1+O(2^{-m(z)}))
.$$ Feeding this into
\eqref{divine aria_Atto Primo: Scena 6 - No. 5 Aria (Paisiello)  Nina O Sia La Pazza Per Amore} concludes the proof. \end{proof}
We obtain the leading constant  
$$ c_r=2^{r-1}
\sum_{k=0}^r {r\choose k}
a(k,r)
\mathfrak S(k,r)
,$$ 
where 
$$a(k,r):= 
 \sum_{\substack{ \b s,\b s'\in \{-1,1\}^r\\ 
\prod_{i=1}^r (s_i,s'_i)_\R=1 }}  
\sum_{\substack{ \c A \subset \{1,\ldots, r\}
\\
 \#\c A=k  }} 
\l(\prod_{i \in \c A} (s_i,s'_i)_\R\r)$$
and 
$$ \mathfrak S(k,r):=((2/9)^{k}+(2/9)^{r-k} ) 
\prod_{p\neq 2 } 
 \l(1-\frac{1}{p}\r)^{-r} 
\frac{1}{2}
\l( \l(  1+\frac{1}{p} \r)^{-2 k}+
\l(  1+\frac{1}{p} \r)^{-2 (r-k)} \r) .$$

\begin{lemma} \label{eq:substack} For all integers $k,r$ with 
$0\leq k \leq r $ we have 
$$a(k,r)=\frac{1}{2}  {r \choose k }
 \l( 4^{r-k} 2^k + 2^{r-k} 4^k  \r) .$$\end{lemma}\begin{proof}
For any $c_i \in \mathbb C$  
the coefficient of $t^k$ in  $\prod_{i=1}^r (1+ c_i t)$ is 
$$ \sum_{\substack{ \c A \subset \{1,\ldots, r\} \\ \#\c A=k }}
\prod_{i\in \c A} c_i.$$ Using this for $c_i=(s_i,s'_i)_\R$
we infer that $$\sum_{\substack{ \c A \subset \{1,\ldots, r\}
\\ \#\c A=k  }} 
\l(\prod_{i \in \c A} (s_i,s'_i)_\R\r)$$ is the coefficient of 
$t^k$ in  $\prod_{i=1}^r (1+ (s_i,s'_i)_\R t)$. Hence, 
$a(k,r)$ is the coefficient of 
$t^k$ in $$\sum_{\substack{ \b s,\b s'\in \{-1,1\}^r\\ 
\prod_{i=1}^r (s_i,s'_i)_\R=1 }}  
\prod_{i=1}^r (1+ (s_i,s'_i)_\R t) =\frac{P_1(t)+P_2(t)}{2}, $$ where 
$$P_1(t):= \sum_{\b s,\b s'\in \{-1,1\}^r}  
\prod_{i=1}^r (1+ (s_i,s'_i)_\R t), \ \ \ 
P_2(t):=   \sum_{ \b s,\b s'\in \{-1,1\}^r  }  
\prod_{i=1}^r (s_i,s'_i)_\R
(1+ (s_i,s'_i)_\R t).$$ Both polynomials
split as products of sums over pairs $(s_i,s'_i)$.
In particular, $$P_1(t)=\prod_{i=1}^r
\sum_{(s_i, s'_i)\in \{-1,1\}^2 }  
 (1+ (s_i,s'_i)_\R t) =(4+2t)^r
=\sum_{k=0} {r \choose k } 4^{r-k} 2^k t^k $$ and 
\[P_2(t)=\prod_{i=1}^r
\sum_{(s_i, s'_i)\in \{-1,1\}^2 }  
(s_i,s'_i)_\R  (1+ (s_i,s'_i)_\R t) = (2+4t)^r =\sum_{k=0}
{r \choose k } 2^{r-k} 4^k t^k.\qedhere\]
\end{proof}
This concludes the proof of the explicit expression for 
$c_r$.

We lower bound $c_r$ by the term coming from    
$k=\lfloor r/2\rfloor$. Stirling's inequalities   
show that  
$$  { r\choose \lfloor r/2\rfloor }\geq  \frac{2^r}{2\sqrt{r}} .$$
In addition, note that 
$\min\{x^{k}+x^{r-k}: k\in \R\cap [0,r]\} = 2x^{r/2}$ holds for 
$x>0$. Hence,  
$$ c_r \geq   
\frac{2^{r-2}}{\pi^r}
  \frac{4^r}{4r}
(2^r\cdot 2 \cdot  2^{r/2})
(2\cdot  (2/9)^{r/2} ) 
\prod_{p\neq 2 } 
 \l(1-\frac{1}{p^2}\r)^{-r}    
 = 
  \frac{  \zeta(2)^r }{4r}
 \l( \frac{2^3}{\pi}\r)^r 
   = 
  \frac{1 }{4r}
 \l( \frac{4 \pi }{  3 }\r)^r .$$

\end{document}